\documentclass[11pt,reqno]{article}
\usepackage[english]{babel}
\usepackage{graphicx}
\usepackage[centertags]{amsmath}
\usepackage{amsfonts,amssymb,amsthm}
\usepackage[]{hyperref}
\usepackage[latin1]{inputenc}
\usepackage{xcolor} 

\numberwithin{equation}{section}

\theoremstyle{plain}
\newtheorem{theorem}{Theorem}[section]
\newtheorem{proposition}[theorem]{Proposition}
\newtheorem{lemma}[theorem]{Lemma}

\theoremstyle{definition}
\newtheorem{remarkbase}[theorem]{Remark}
\newenvironment{remark}{\pushQED{\qed}\remarkbase}{\popQED\endremarkbase}

\newtheoremstyle{hypstyle}{}{}{}{}{\bfseries}{.}{ }%
{\thmname{#1}\thmnumber{ (H#2)}\thmnote{}}

\theoremstyle{hypstyle}

\def\notina[#1]#2{\begingroup\def\thefootnote{\fnsymbol{footnote}}\footnote[#1]{#2}\endgroup}

\newcommand{\N}{{\mathbb N}}
\newcommand{\R}{{\mathbb R}}
\newcommand{\C}{{\mathbb C}}
\newcommand{\Z}{{\mathbb Z}}

\newcommand{\T}{{\mathbb T}}

\newcommand{\mD}{\mathcal{D}}
\newcommand{\mE}{\mathcal{E}}

\newcommand{\mN}{\mathcal{N}}
\newcommand{\mP}{\mathcal{P}}

\newcommand{\mZ}{\mathcal{Z}}
\newcommand{\tA}{\mathtt{A}}
\newcommand{\tB}{\mathtt{B}}

\renewcommand{\a}{\alpha}
\renewcommand{\b}{\beta}
\newcommand{\g}{\gamma}
\renewcommand{\d}{\delta}

\newcommand{\e}{\varepsilon}
\newcommand{\ph}{\varphi}
\newcommand{\lm}{\lambda}

\newcommand{\om}{\omega}

\newcommand{\s}{\sigma}

\renewcommand{\th}{\vartheta}

\renewcommand{\Im}{\mathrm{Im}\,}
\renewcommand{\Re}{\mathrm{Re}\,}

\newcommand{\gr}{\nabla}

\newcommand{\pa}{\partial}

\newcommand{\ov}{\overline}

\newcommand{\bcb}{\begin{color}{blue}}
\newcommand{\bcr}{\begin{color}{red}}
\newcommand{\ec}{\end{color}}

\title{Effective chaos for the Kirchhoff equation on tori}

\author{\small{Pietro Baldi, Filippo Giuliani, Marcel Guardia, Emanuele Haus}}

\date{} 

\begin{document}

\maketitle

\begin{abstract}
We consider the Kirchhoff equation on tori of any dimension and we construct solutions 
whose Sobolev norms oscillates in a chaotic way on certain long time scales. 
The chaoticity is encoded in  the time between oscillations of the norm, which can be chosen in any prescribed way. This phenomenon, that we name as \emph{effective chaos} (it occurs over a long, but finite, time scale), is consequence of the existence of symbolic dynamics for an effective system. Since the first order resonant dynamics has been proved to be essentially stable, we need to perform a second order analysis to find an effective model displaying chaotic dynamics. More precisely, after some reductions, this model behaves as two weakly coupled pendulums.
\end{abstract}

\tableofcontents

\section{Introduction and main result} 
\label{sec:intro}



We consider the Kirchhoff equation 
\begin{equation} \label{K Om}
\pa_{tt} u - \Delta u \Big( 1 + \int_{\T^d} |\gr u|^2 \, dx \Big) = 0 
\end{equation}
on the torus $\T^d$, 
$\T := \R / 2\pi \Z$, in any dimension $d \geq 1$ 
(periodic boundary conditions), 
where the unknown $u=u(t,x)$, $x\in\T^d$, is a real-valued function. 

Equation \eqref{K Om} was first introduced by Kirchhoff \cite{Kirchhoff.1876} in 1876, 
to model nonlinear transverse oscillations of strings and plates ($d=1,2$).
It is a quasilinear wave equation, 
with cubic, nonlocal nonlinearity and Hamiltonian structure.
Given its physical relevance, equation \eqref{K Om} has been largely studied along the years;
nonetheless, 
its study 
is still 
challenging, because several basic questions remain open.

While it has long been known 
(Dickey \cite{Dickey.1969}, Arosio-Panizzi \cite{Arosio.Panizzi.1996}) 
that the Cauchy problem for \eqref{K Om} is locally wellposed 
with initial data $(u(0), \pa_t u(0))$ 
in the Sobolev space $H^{\frac32}(\T^d, \R) \times H^{\frac12}(\T^d, \R)$, 
it is still an open problem whether the solutions 
with initial data of any given Sobolev regularity are global in time or not. 
In particular, it is not even known if $C^\infty$ (or even Gevrey) initial data 
of small amplitude produce solutions that are global in time. 
For initial data in analytic class, instead, 
global wellposedness is known since the work of Bernstein \cite{Bernstein.1940} in 1940.

Moreover, below the regularity threshold $H^{\frac32}\times H^{\frac12}$, 
neither local wellposedness nor illposedness have been established.  
A partial, interesting result in this direction has been recently obtained 
by Ghisi and Gobbino \cite{Ghisi.Gobbino.2022}. 

More general questions regard the lifespan of the solutions 
and their behavior as time evolves, at least close to the equilibrium $u=0$.
First of all, as a consequence of the linear theory, for initial data of size $\e$ 
in $H^{\frac32} \times H^{\frac12}$, the existence of the solution is guaranteed 
at least for a time of the order $\e^{-2}$. 
Since \eqref{K Om} is a quasilinear equation, 
it is not a priori obvious that one can obtain better estimates. 
For instance, in the well-known example by Klainerman and Majda \cite{Klainerman.Majda.1980} 
{\em all} nontrivial space-periodic solutions of size $\e$ 
blow up in a time of order $\e^{-2}$. 
On the other hand, in the papers \cite{Baldi.Haus.Nonlin}, \cite{Baldi.Haus.JDDE}, 
\cite{Baldi.Haus.SIAM}, using techniques from the normal form theory,
it is proved that for the Kirchhoff equation the situation is more favorable. 
More precisely, in \cite{Baldi.Haus.Nonlin}, 
performing one step of quasilinear normal form, 
it is proved that the lifespan of all solutions of small amplitude 
is at least of order $\e^{-4}$. This is a consequence of the fact that 
the only resonant cubic terms that cannot be erased in the first step of normal form 
give no contribution to the energy estimates. 
In \cite{Baldi.Haus.JDDE} the second step of quasilinear normal form 
is computed, 
and it is proved that there are resonant terms of degree five that cannot be erased 
and that give a nontrivial contribution to the time evolution of Sobolev norms. 

This is a starting point for describing interesting long-time dynamics for the Kirchhoff equation. 
%
%
%
%
%
The qualitative behavior of solutions of the Kirchhoff equation over long-time scales is poorly understood, even for small, compactly Fourier supported initial data, which obey to finite dimensional systems. 

Broadly speaking, for the dynamics of small data we can look for two different types of regimes: 
\begin{itemize}
\item Stable regime: this is the case in which the long-time behavior of Fourier modes 
resembles the dynamics of the linearized equation, namely the energy of the modes 
remains almost constant over long time. We mention \cite{Baldi.Haus.SIAM} where 
stable motions of equation \eqref{K Om} are obtained 
for a suitable set of nonresonant initial data, 
for which  the effect of the resonant terms of degree five remain small 
on a longer timescale of order $\e^{-6}$.  
We also mention \cite{Baldi.2009} and \cite{Corsi.Montalto.2018},
where the existence of invariant tori is proved for a forced version of \eqref{K Om}.

\item  Unstable regime: in this case the nonlinear terms lead to a new type of dynamics, 
very different from the linear one. 
Of particular interest 
is understanding how the nonlinear effects create exchanges of energy among different modes. 
\end{itemize}

Concerning the unstable regime, 
some remarkable results in literature regard the ``energy cascade'' 
for nonlinear Schr\"odinger equations, 
where the energy travels from 
%
low to high modes (or vice versa), 
in strong connection with the weak turbulence theory. 
Such phenomenon, which can be measured in terms of an arbitrarily large growth of Sobolev norms, 
was considered by Bourgain one of the most important problems in Hamiltonian PDEs, 
see \cite{Bourgain00b}, 
and also \cite{CKSTT.2010,  GHHMP.2018, Guardia.Haus.Procesi.2016, 
Guardia.Kaloshin.2015, HaniPTV15}.

In the unstable regime, other interesting dynamical behaviors 
are also based on the mechanism of energy exchange among Fourier modes.
Such exchanges can be recurrent (i.e., periodic or quasi-periodic in time) or chaotic. 
Recurrent energy exchanges are obtained, for instance, in 
\cite{GPT, Grebert.Thomann.2012, GrebertV11, Haus.Procesi.CMP.2017, HT}. 
To the best of our knowledge, the only paper in literature 
in which chaotic exchanges of energy are constructed for PDEs is \cite{GGMP}. 
In \cite{GGMP} the authors consider cubic wave and beam equations 
and prove the existence of solutions essentially Fourier supported 
on a finite number of resonant modes that exchange energy among themselves in a chaotic way. 
The chaoticity reflects in the fact that it is possible 
to provide energy exchanges among modes at a sequence of prescribed times 
(randomness of exchanging time) 
or among modes belonging to a prescribed resonant tuple (randomness of active and inactive modes). 

Both in \cite{GGMP} and in the present paper 
the existence of chaotic solutions 
is due to the presence of chaotic dynamics 
for the normal form of the equation, up to a certain degree. 
More precisely, the normalized system 
leaves invariant a finite dimensional subspace; 
then the chaotic behavior arises 
from the existence of a Smale horseshoe, 
which gives rise to symbolic dynamics.
The orbits of the normalized system are globally defined in time, 
and the chaotic behavior is displayed for an infinitely long time. 
However, this does not imply the existence of chaotic solutions 
of the full PDE for an infinitely long time.
The chaotic behavior for the full PDE is obtained by proving 
the vicinity of certain solutions of it 
to the chaotic orbits of the normalized system, 
and this approximation only holds over a long, but finite, time interval. 
We call this behavior \emph{effective chaoticity}, 
in the sense that the dynamics behaves as chaotic in rather long time scales 
(in analogy to the stability over long time scales, 
often called \emph{effective stability} in Hamiltonian dynamics).

\subsection{Main result}

We denote by $\N_0$ the set $\{ 0, 1, 2, \ldots \}$ of nonnegative integers.
The next theorem, which shows the existence of solutions of the Kirchhoff equation displaying chaotic-like, small amplitude, oscillations in the Sobolev norms,  is the main result of the paper.

\begin{theorem} \label{thm:short}
There exist universal positive constants $M, \tau, \e_*, C, r_0, b, K, K_0$ 
with the following property. 
Let $d \geq 1$. 
For every sequence $(m_j)_{j \in \N_0} = (m_0, m_1, m_2, \ldots)$ 
of integers such that $m_j \geq M$ for all $j \in \N_0$, 
there exists a sequence 
\[
0 = s_0 < \bar s_0 < s_1 < \bar s_1 < s_2 \leq \bar s_2 < \ldots, 
\quad \ 
s_{j+1} = s_j + \tau (m_j + \theta_j), 
\quad \ 
0 \leq \theta_j < 1,
\]
such that for every $\e \in (0, \e_*]$ 
there exists a solution $u(t,x)$ of the Kirchhoff equation \eqref{K Om} on $\T^d$,
global in time, with finite Fourier support, 
whose norm 
\[
\mN(t) := \Big( \| u(t) \|_{H^{\frac32}(\T^d)}^2 
+ \| \pa_t u(t) \|_{H^{\frac12}(\T^d)}^2 \Big)^{\frac12}
\]
satisfies 
\[
\mN(t) \leq C \e \quad \ \forall t \in \R
\]
and it oscillates around the central value 
$A_\e := \e + \e^2 r_0$ 
with oscillations described in terms of 
the amplitude $B_\e := \e^2 r_0$ 
and the error $\delta_\e := \frac{1}{10} B_\e$
as
\begin{alignat*}{3}
- \delta_\e 
& \leq \mN(t) - A_\e 
\leq B_\e + \delta_\e 
\quad && \forall t \in I_j = [t_j, \bar t_j],
\qquad & 
\max_{t \in I_j} \mN(t) - A_\e
& \geq B_\e - \delta_\e,
\notag \\
- B_\e - \delta_\e 
& \leq \mN(t) - A_\e
\leq \delta_\e 
\quad && \forall t\in E_j = [\bar t_j, t_{j+1}],
\qquad & 
\min_{t \in E_j} \mN(t) - A_\e
& \leq - B_\e + \delta_\e,
\end{alignat*}
where 
\[
t_j = \frac{s_j}{b \e^3}, \quad \ 
\bar t_j = \frac{\bar s_j}{b \e^3},
\]
for all intervals $I_j, E_j$ contained in the time interval $[0, T_\e]$, where 
\[
T_\e = K \e^{-3} \log(\e^{-1}).
\] 
One has $I_j, E_j \subset [0, T_\e]$ 
for all $j = 0, \ldots, N$, where the integer $N$ satisfies 
\[
\sum_{j=0}^N m_j \leq K_0 \log(\e^{-1}). 
\]
\end{theorem}

In other words, Theorem \ref{thm:short} says that, 
around the equilibrium $ u = 0$, 
the Kirchhoff equation possesses solutions 
whose norm $\mN(t)$
exhibits oscillations 
that follow any prescribed sequence of times 
on the time interval $[0, T_\e]$, 
and the number $N$ of oscillations within that interval, 
or more generally the sum of the time lengths of the oscillations, 
is arbitrarily large for $\e$ small enough. 
These oscillations can also be seen as a {\em chaotic-like modulation} of a stable motion, 
meaning that the oscillating solutions are of size $\e$, they are $\e^2$-close to effectively stable solutions (over long time scales), but they exhibit chaotic-like exchanges of size $\e^2$ between the amplitude  of different Fourier modes.

\begin{remark} 
\label{rem:support} 
The solution $u(t,x)$ in Theorem \ref{thm:short} is Fourier supported 
on the set $\{ k \in \Z^d : |k| \in \{ \a_1, \a_2, \a_3, \a_4 \} \}$, 
where 
\[
\a_1 = m, \quad \ 
\a_2 = m + p, \quad \ 
\a_3 = 2m + p, \quad \ 
\a_4 = 3m + 2p,
\]
and $m,p$ are integers with $2 \leq m < p$ and ratio $\s = m/p \leq \s_*$, 
where $\s_*$ is a universal constant.

In fact, the ratio $\s$ is the perturbation parameter we use in the entire construction.
In principle, the constant $M$ in Theorem \ref{thm:short} 
depends on the ratio $\s = m/p$ and it is of the order $M \sim \log(\s^{-1})$, 
see \eqref{log.M.0}.
Theorem \ref{thm:short} is stated after fixing $m,p$ with $m=2$ and 
$p$ the minimum integer such that $p > 2$ and $2/p \leq \s_*$.
\end{remark}

\begin{remark}
\label{rem:norm.choice}
In Theorem \ref{thm:short} the Sobolev norm $\mN(t)$ is used 
to describe the transfer of energy between Fourier modes, 
because $H^{\frac32}(\T^d) \times H^{\frac12}(\T^d)$ 
is the space of the standard local wellposedness for the Kirchhoff equation. 
Since the solution $u(t,x)$ in Theorem \ref{thm:short} 
has a fixed, finite Fourier support for all times, 
all the Sobolev norms of $(u, \pa_t u)$ are equivalent,  
and all are equally able to describe the chaotic transfer of energy among the Fourier modes 
--- all except the norm of the energy space $H^1(\T^d) \times L^2(\T^d)$, 
which corresponds to a conserved quantity of the 
approximating system that we use in the construction;
see Remark \ref{rem:norm.12.prime.int}.
\end{remark}

\begin{remark}
\label{rem:1/10}
The factor $1/10$ in the definition of $\d_\e$ in Theorem \ref{thm:short} 
comes from an arbitrary choice. 
We could replace 1/10 by any other positive number; 
in that case, the constants $M, \e_*, C, r_0, b, K, K_0$ must be chosen accordingly.
\end{remark}

\begin{remark}
\label{rem:neg.times}
For simplicity, Theorem \ref{thm:short} and its proof are stated 
entirely in terms of nonnegative times. 
However, with only minor changes, one proves that the result holds 
over the time interval $[- T_\e, T_\e]$.
\end{remark}

\begin{remark} 
\label{rem:up.down}
Adapting the formulation of the symbolic dynamics for the approximating system 
(see Proposition \ref{prop:Smale}), 
one can prove an alternative version of Theorem \ref{thm:short}, 
where the prescribed random behavior of the norm $\mN(t)$ 
is not only given by the sequence of the time lengths of its oscillations, 
but also by any sequence $(a_0, a_1, a_2, \ldots)$ with $a_j \in \{ 0, 1\}$
prescribing the ordered sequence of ``up'' and ``down'' movements of $\mN(t)$.
In that case, $\mN(t)$ still makes oscillations of order $\e^2$  
around a central value of order $\e$, varying in a range, say, $[\ell_\e, h_\e]$; 
the difference with respect to Theorem \ref{thm:short} is that, 
in the $j$-th time interval, 
$\mN(t)$ get close to the low value $\ell_\e$, 
and it remains in the slightly enlarged lower half of the range, 
if $a_j = 0$, 
while $\mN(t)$ get close to the high value $h_\e$,
and it remains in the slightly enlarged upper half of the range, 
if $a_j = 1$.

In other words, around the equilibrium $u = 0$, 
the Kirchhoff equation possesses solutions 
whose norm $\mN(t)$ exhibits oscillations 
that follow any prescribed sequence of ``up'' and ``down''
on the time interval $[0, T_\e]$. 
\end{remark}

\begin{remark}
\label{rem:smaller.shorter}
The result in \cite{Baldi.Haus.Nonlin} shows that there are \emph{no} transfers 
of energy of size $\e$ between Fourier spheres in a time interval of length $\e^{-4}$. 
This could make one think that, on such a time scale, 
between Fourier spheres there are no energy transfers at all. 
Theorem \ref{thm:short} shows that this is not true; 
in particular, it proves the existence of chaotic transfers of energy 
of smaller size on a shorter time scale, i.e., 
transfers of size $\e^2$ on a time scale $\e^{-3} \log(\e^{-1})$.
\end{remark}

\subsection{Main ideas of the proof} 
The main steps of the proof of Theorem \ref{thm:short} can be summarized as follows:
\begin{enumerate}
\item Derive an effective resonant model for small solutions of \eqref{K Om}. 
This reduced system is obtained by using normal form arguments and introducing 
some ``macroscopic'' variables describing the collective behavior of Fourier frequencies
with the same modulus.

\item 
Show that, choosing carefully a finite set of Fourier frequencies, 
one can make the effective system nearly integrable. 

\item Prove the existence of chaotic dynamics (a Smale horseshoe) for the effective system.

\item Show that certain solutions of the Kirchhoff equation \eqref{K Om} follow closely those 
in the Smale horseshoe of the effective system for a sufficiently long time interval. 
 \end{enumerate}


The effective system is obtained with a 
normal form analysis.
To this end, in Section \ref{sec:effectivedyn}, we perform two steps of 
\emph{quasilinear normal form} (following \cite{Baldi.Haus.Nonlin,Baldi.Haus.JDDE}), 
and introduce a set of special variables, 
found in \cite{Baldi.Haus.JDDE}, 
which allow to reduce the dimension of the problem 
without losing information on the time evolution 
of the Sobolev norm of the solution. 
Thus, the resulting reduced model can be seen as a ``macroscopic'' effective system, 
where we do not distinguish the evolution of the energy of each single Fourier mode. 

The reduction to a finite dimensional effective system is done in Section \ref{sec:2dofHS}. 
We restrict the Fourier support to two coupled resonant triplets. 
The space of functions supported on these modes is invariant for equation \eqref{K Om},
thanks to the particular form of its nonlinearity. 
Relying on symmetries of the problem, we are able to further reduce the model to obtain 
a four dimensional system.

The next step is to construct chaotic motions for such a system. 
Since we rely on perturbative techniques, we want the system to be nearly integrable;  
this is obtained by choosing resonant triplets with Fourier modes 
as explained in Remark \ref{rem:support}.
In particular, the system behaves as a pair of weakly coupled pendulums. 

Then, in Section \ref{sec:chaos.2.pendulums}, 
we apply the classical Poincar\'e-Melnikov theory \cite{Melnikov63}  
to prove that the system has a hyperbolic periodic orbit with transverse homoclinic orbits. 
By the classical Smale-Birkhoff Theorem, this implies the existence of a Smale horseshoe, 
which is a hyperbolic invariant set with symbolic dynamics.
Note that the set is invariant 
and therefore one can describe the dynamics of its orbits \emph{for all times}.

Finally, it remains to translate the dynamics of the effective system 
to the original equation \eqref{K Om}. 
We prove that there exist solutions of the full PDE that follow closely those of the effective system. 
Even if the approximation argument is done through a Gronwall estimate (see Section \ref{sec:gronwall}),
this is a rather delicate procedure.
Indeed, since we have performed several reductions, rescalings, and two steps of normal form, 
we have to ensure that the solutions of the Kirchhoff equation 
shadowing those of the effective system satisfy all the required constraints 
over a sufficiently long time scale. 
This final part of the proof is done in Sections \ref{sec:back} and \ref{sec:gronwall}.

\bigskip

The general strategy of the proof is similar to the one 
developed in \cite{GGMP} for the cubic wave and beam equations. 
The proof of Theorem \ref{thm:short}, however, 
is based on a higher order normal form analysis, which is needed to consider systems 
which are integrable at first order. 
This is the typical situation for PDEs on one-dimensional spatial domains. 
Indeed, resonant Hamiltonian monomials of low degree, which provide the dominant dynamics 
close to the origin, usually do not change drastically the Fourier actions (and so, the Sobolev norms). 
Main examples are given by the KdV, Klein-Gordon and Schr\"odinger equations 
and pure gravity water waves equation in infinite depth, 
under Dirichlet or periodic boundary conditions. 
This is somewhat the case also for equation \eqref{K Om}, 
even if the spatial domain is the torus $\T^d$ of any dimension $d \geq 1$,
and even if the integrability property of the equation at the cubic order 
only holds for the macroscopic variables. 
This makes the implementation of the above strategy rather delicate. 
One of the issues in performing this kind of analysis is that interesting instability phenomena 
only occur after a longer time.

Another relevant difference with respect to the equations considered in \cite{GGMP} 
is that equation \eqref{K Om} is quasilinear, 
namely the nonlinearity contains derivatives of the same order  
as the linear part. This fact is not trivial, because, 
even if one is able to construct normal form transformations 
for the quasilinear equation \eqref{K Om} (as done in \cite{Baldi.Haus.Nonlin, Baldi.Haus.JDDE}), 
here one has to be able to provide a result of approximation 
between the effective model and the full PDE for a long-time scale. 
This requires to consider an equation for the difference 
of a special orbit of the effective system and a solution of \eqref{K Om}. 
This equation is quasilinear itself and presents a time-dependent linear part. 
For such equation one has to provide a result of long-time stability.

Another difference with respect to \cite{GGMP} 
regards a quantitative aspect in the energy exchange between Fourier frequencies: 
in \cite{GGMP} a large portion of the energy transfers between Fourier frequencies 
having similar modulus; here, on the contrary, a very small portion of the energy 
transfers between Fourier frequencies of modulus $\a_1, \a_2, \a_3, \a_4$ 
(see Remark \ref{rem:support}), where $\a_2, \a_3, \a_4$ are much larger than $\a_1$.

\subsection*{Acknowledgments}

P.B. and E.H. are supported by the Italian Project 
PRIN 2020XB3EFL \emph{Hamiltonian and dispersive PDEs}.
F.G. and E.H. have received funding from INdAM-GNAMPA, 
Project CUP\_E55F22 000270001.
M.G. is supported by the European Research Council (ERC) 
under the European Union's Horizon 2020 research and innovation programme 
(grant agreement No. 757802). 
M.G. is also supported by the Catalan Institution for Research and Advanced Studies 
via an ICREA Academia Prize 2019. 
This work is also supported by the Spanish State Research Agency, 
through the Severo Ochoa and Mar\'ia de Maeztu Program for Centers and Units of Excellence 
in R\&D (CEX2020-001084-M).

\section{Effective dynamics for the Kirchhoff equation}\label{sec:effectivedyn}

In this section we recall how the ``macroscopic'' quantities 
$S_\lm, B_\lm$ in \eqref{def.SB} are derived from the Kirchhoff equation \eqref{K Om},
starting with the normal form procedure. 


\subsection{A quasilinear partial normal form}
\label{sec:second step}

Written as a first order evolution equation, \eqref{K Om} becomes 
\begin{equation} \label{p1}
\begin{cases} 
\pa_t u = v, \\ 
\pa_t v = \big( 1 + \int_{\T^d} |\gr u|^2 dx \big) \Delta u.
\end{cases}
\end{equation}
It is proved in \cite{Baldi.Haus.Nonlin} and \cite{Baldi.Haus.JDDE} 
(see also the shorter, unified description in the Appendix A of \cite{Baldi.Haus.SIAM})
that system \eqref{p1} can be transformed, 
after two steps of a quasilinear, partial normal form procedure, 
into another system, where the cubic and the quintic terms are in normal form
(up to harmless terms that do not contribute to energy estimates). 
More precisely, it is proved that, 
renaming $(\tilde u, \tilde v)$ the original, ``physical'' variables of system \eqref{p1},  
with the change of variable $(\tilde u, \tilde v) = \Phi(u,v)$,
system \eqref{p1} becomes 
\begin{equation} \label{3101.16}
\pa_t (u,v) = W(u,v),
\end{equation}
where 
\begin{equation} \label{W decomp}
W(u,v) = (1 + \mP(u,v)) 
\big( \mD_1(u,v) + \mZ_3(u,v) + \mZ_5(u,v) \big) 
 + W_{\geq 7}(u,v).
\end{equation}
The unknown $(u,v)$ for the transformed system \eqref{3101.16} 
is a pair of complex conjugate functions with zero average over $\T^d$.  
The term $(1 + \mP(u,v))$ is a scalar multiplicative factor, close to $1$,
depending on $(u,v)$, and it is a function of time, 
independent of the space variable $x$.
Also, $\mD_1$ is the linear operator 
$\mD_1 (u,v) := (- i |D_x| u,  i |D_x| v)$,
where $|D_x|$ is the Fourier multiplier
$e^{ik \cdot x} \mapsto |k| e^{ik \cdot x}$, $k \in \Z^d$.
Next, $\mZ_3$ is the cubic resonant operator 
$\mZ_3(u,v) = ((\mZ_3)_1(u,v) , (\mZ_3)_2(u,v))$
with components
\begin{align*}
(\mZ_3)_1(u,v) 
:= - \frac{i}{4} \sum_{\begin{subarray}{c} j,k \in \Z^d \setminus\{0\} \\ |j| = |k| \end{subarray}} 
u_j u_{-j} |j|^2 v_k e^{ik \cdot x},
\quad \ 
(\mZ_3)_2(u,v) 
:= \frac{i}{4} \sum_{\begin{subarray}{c} j,k \in \Z^d \setminus\{0\} \\ |j| = |k| \end{subarray}}
v_j v_{-j} |j|^2 u_k e^{ik \cdot x},
\end{align*}
where $u_j, v_j$ are the Fourier coefficients of $u,v$.
The entire term 
$(1 + \mP(u,v)) (\mD_1(u,v) + \mZ_3(u,v))$ 
gives no contribution to the energy estimates.  
The term $\mZ_5$ is the resonant quintic operator 
$
\mZ_5(u,v) = 
((\mZ_5)_1(u,v) , (\mZ_5)_2(u,v)),
$
where 
\begin{align}
(\mZ_5)_1(u,v) 
& = \frac{i}{32} \sum_{\begin{subarray}{c} j,\ell,k \in \Z^d \setminus\{0\} \\ |j| = |\ell| \end{subarray}} 
u_j u_{-j} v_\ell v_{-\ell} u_k e^{ik \cdot x} 
|j|^2 |\ell|^2 
\Big( \frac{1}{|j|+|k|} - \frac{(1 - \d_{|\ell|}^{|k|})}{|\ell| - |k|} \Big)
\notag \\ & \quad 
+ \frac{3 i}{32} 
\sum_{\begin{subarray}{c} j,\ell,k \in \Z^d \setminus\{0\} \\ |k| = |j| + |\ell| \end{subarray}} 
u_j u_{-j} u_\ell u_{-\ell} v_k e^{ik \cdot x} 
|j| |\ell| |k|
\notag \\ & \quad 
+ \frac{i}{16}  \sum_{\begin{subarray}{c} j,\ell,k \in \Z^d \setminus\{0\} \\ |j| = |k| \end{subarray}} 
u_j u_{-j} u_\ell v_{-\ell} v_k e^{ik \cdot x} 
|j|^2 |\ell|
\Big( 6 + \frac{|\ell|}{|\ell| + |j|}
+ \frac{|\ell| (1 - \d_{|\ell|}^{|j|}) }{|\ell| - |j|} \Big)
\notag \\ & \quad 
+ \frac{3 i}{16} \sum_{\begin{subarray}{c} j,\ell,k \in \Z^d \setminus\{0\} \\ |k| = |j| - |\ell| \end{subarray}} 
u_j u_{-j} v_\ell v_{-\ell} v_k e^{ik \cdot x} |j| |\ell| |k|,
\label{W5 comp 1 after NF}
\end{align}
and $(\mZ_5)_2 (u,v)$ is obtained from $(\mZ_5)_1 (u,v)$ by complex conjugation.
For the coefficients in \eqref{W5 comp 1 after NF} 
we adopt the convention that $\frac00=0$,
where $\d$ is the usual Kronecker delta.
The term $W_{\geq 7}(u,v)$ in \eqref{W decomp} 
contains only terms of homogeneity at least 7 in $(u,v)$, 
and it is estimated in \cite{Baldi.Haus.JDDE} and \cite{Baldi.Haus.SIAM}.

The map $\Phi$ that transforms \eqref{p1} into \eqref{3101.16} 
is obtained by composition, and it is  
\begin{equation} \label{def.Phi}
\Phi = \Phi_1 \circ \Phi_2 \circ \Phi_3 \circ \Phi_4 \circ \Phi_5.
\end{equation}
The maps $\Phi_1$ and $\Phi_2$ are simply 
the linear operators that symmetrize and diagonalize 
the linear part of system \eqref{p1} (i.e., the linear wave equation), 
see \eqref{f.g.q.p.tilde.u.tilde.v}. 
%
The map $\Phi_3$ is a nonlinear, preparatory transformation, 
which is required because the problem is quasilinear.
The map $\Phi_4$ is the transformation of the first step of the normal form procedure, 
and $\Phi_5$ is the one of the second step.
The explicit expressions of $\Phi_3$ and $\Phi_4$ are given in \cite{Baldi.Haus.Nonlin} 
and the one of $\Phi_5$ in \cite{Baldi.Haus.JDDE}. 
Unlike $\Phi_1$ and $\Phi_2$, 
the transformations $\Phi_3, \Phi_4, \Phi_5$ are all close to the identity map. 
In the present paper we do not use the explicit formula of $\Phi_3, \Phi_4, \Phi_5$,  
but only the following properties of their composition. 

For $s \in \R$, let $H^s_0(\T^d,\C)$ be the Sobolev space of 
zero average, complex-valued functions 
\[
H^s_0(\T^d,\C) := \{ u : \T^d \to \C : u_0 = 0, \ \| u \|_s < \infty \},
\quad \ 
\| u \|_s^2 := \sum_{k \in \Z^d \setminus \{0\}} |u_k|^2 |k|^{2s},
\]
where $u_k$, $k \in \Z^d$, are the Fourier coefficients of $u$, 
and let 
\[
H^s_0(\T^d,c.c.) := \{ (u,v) : u,v \in H^s_0(\T^d,\C), \ v = \overline{u} \}
= \{ (u, \overline{u}) :  u \in H^s_0(\T^d,\C) \},
\]
where the notation ``$c.c.$'' reminds that they 
are pairs of complex conjugate functions. 
Given $d \in \N$, let
\begin{equation} \label{def.m1}
m_1 := 1 \quad \text{if } d=1; \qquad 
m_1 := 2 \quad \text{if } d \geq 2.
\end{equation}

\begin{lemma}[From Lemma 2.9 of \cite{Baldi.Haus.SIAM}]
\label{lemma:Lemma.2.9.SIAM}
There exist universal constants $\d, C > 0$ such that 
for all $(u,v) \in H^{m_1}_0(\T^d, c.c.)$ in the ball $\| u \|_{m_1} \leq \d$,
for all $k \in \Z^d$, the $k$-th Fourier coefficient $f_k = \overline{g_{-k}}$ 
of $(f,g) := (\Phi_3 \circ \Phi_4 \circ \Phi_5)(u,v)$ satisfies 
\begin{equation} \label{2.77.SIAM}
|f_k - u_k| 
\leq C \| u \|_{m_1}^2 (|u_k| + |u_{-k}|).
\end{equation}
\end{lemma}

From \eqref{2.77.SIAM} it follows that 
$\| f - u \|_s \leq 2 C \| u \|_{m_1}^2 \| u \|_s$ for all $s \in \R$, 
and that the map $\Phi_3 \circ \Phi_4 \circ \Phi_5$ 
is well defined in the ball $\| u \|_{m_1} \leq \d$, with 
$\| f \|_s \leq \| u \|_s (1 + 2 C \| u \|_{m_1}^2)$, $s \in \R$. 

Inequality \eqref{2.77.SIAM} also implies the invariance of the Fourier support: 
if $u_k = u_{-k} = 0$ for some $k \in \Z^d$, then also $f_k = f_{-k} = 0$, 
and vice versa (since $C \d^2 < 1/2$).

\begin{lemma}[From Lemma 2.3 in \cite{Baldi.Haus.SIAM}] 
\label{lemma:total}
There exist universal constants $\d_1, C_1, C_0 > 0$ with the following properties. 
Let $(u_0, v_0) \in H^{m_1}_0(\T^d, c.c.)$ and 
\begin{equation} \label{ball.mitica}
\| u_0 \|_{m_1} \leq \d_1.
\end{equation}
Then the Cauchy problem of system \eqref{3101.16} with initial condition
\begin{equation} \label{initial.cond.u.v.sec.prepar}
(u(0), v(0)) = (u_0, v_0)
\end{equation} 
has a unique solution 
$(u,v) \in C([0,T_{\mathrm{NF}}], H^{m_1}_0(\T^d, c.c.))$ 
on the time interval $[0,T_{\mathrm{NF}}]$. 
The solution satisfies 
\begin{equation} \label{twice.norm.sec.prepar}
\| u(t) \|_{m_1} \leq C_1 \| u_0 \|_{m_1} \leq \d
\quad \forall t \in [0,T_{\mathrm{NF}}],
\qquad 
T_{\mathrm{NF}} = C_0 \| u_0 \|_{m_1}^{-4},
\end{equation}
where $\d$ is the constant in Lemma \ref{lemma:Lemma.2.9.SIAM}.  
As a consequence, 
for all $t \in [0, T_{\mathrm{NF}}]$ 
the solution $(u(t), v(t))$ remains in the ball 
$\| u(t) \|_{m_1} \leq \d$ where $\Phi_3 \circ \Phi_4 \circ \Phi_5$ is well defined, 
the function 
\begin{equation} \label{tilde.uv.from.uv}
(\tilde u(t), \tilde v(t)) := \Phi(u(t),v(t)) 
\end{equation}
(where $\Phi$ is the map in \eqref{def.Phi}) 
solves the original system \eqref{p1} on the time interval $[0,T_{\mathrm{NF}}]$,
and $\tilde u(t)$ solves the Kirchhoff equation \eqref{K Om} 
on $[0,T_{\mathrm{NF}}]$. 
\end{lemma}

As the notation suggests, $T_{\mathrm{NF}}$ is the existence time 
we obtain by the normal form procedure.
For more details on the map $\Phi$ and on the transformed vector field $W(u,v)$
see \cite{Baldi.Haus.Nonlin}, \cite{Baldi.Haus.JDDE}, \cite{Baldi.Haus.SIAM}.

\subsection{The effective system}

We recall the derivation of the \emph{effective system} 
(or \emph{effective equation}) 
from \cite{Baldi.Haus.JDDE}, \cite{Baldi.Haus.SIAM}. 
Let
\[
\Gamma := \{ |k| : k \in \Z^d, \ k \neq 0 \} 
\subseteq \{ \sqrt{n} : n \in \N \} 
\subset [1, \infty).
\]
For any pair $(u,v) \in L^2(\T^d,c.c.)$ of complex conjugate functions, 
for any $\lm \in \Gamma$ we define 
\begin{equation} \label{def.SB}
S_\lm := 
\sum_{\begin{subarray}{c} k \in \Z^d \\ |k| = \lm \end{subarray}} 
|u_k|^2 
= \sum_{\begin{subarray}{c} k \in \Z^d \\ |k| = \lm \end{subarray}} 
u_k v_{-k}, 
\qquad 
B_\lm := 
\sum_{\begin{subarray}{c} k \in \Z^d \\ |k| = \lm \end{subarray}} 
u_k u_{-k},
\end{equation}
and note that 
\[
S_\lm \geq 0, \qquad 
B_\lm \in \C, \qquad 
|B_\lm| \leq S_\lm.
\] 
The quantity $S_\lm$ is called the ``superaction'' of $u$ on the sphere $|k| = \lm$.
Its evolution on the time interval $[0, T_{\mathrm{NF}}]$ 
remains confined between two multiples of its initial value, 
as is observed in the next lemma.

\begin{lemma}[From Lemma 2.4 of \cite{Baldi.Haus.SIAM}]
\label{lemma:Lemma.2.4.SIAM}
Let $(u_0, v_0) \in H^{m_1}_0(\T^d, c.c.)$, with $u_0$ in the ball \eqref{ball.mitica}.  
Let $(u(t), v(t))$ be the solution of the Cauchy problem 
\eqref{3101.16}, \eqref{initial.cond.u.v.sec.prepar}
on the time interval $[0, T_{\mathrm{NF}}]$,
with $T_{\mathrm{NF}}$ in \eqref{twice.norm.sec.prepar}, 
given by Lemma \ref{lemma:total}. 
For every $t \in [0, T_{\mathrm{NF}}]$, 
let $S_\lm(t)$ be the sum defined in \eqref{def.SB}. 
Then 
\begin{equation} \label{S.like.exp}
C_1 S_\lm(0) \leq S_\lm(t) \leq C_2 S_\lm(0)
\end{equation}
for all $t \in [0, T_{\mathrm{NF}}]$,
for all $\lm \in \Gamma$, 
where $C_1, C_2 > 0$ are universal constants.
\end{lemma}

By \eqref{S.like.exp}, for every $\lm \in \Gamma$, 
either $S_\lm(t) > 0$ for all $t \in [0, T_{\mathrm{NF}}]$, 
or $S_\lm(t) = 0$ for all $t \in [0, T_{\mathrm{NF}}]$.
Hence, we decompose $\Gamma$ as the disjoint union of 
\begin{equation} \label{def.Gamma.0.1.general}
\Gamma_0 := \{ \lm \in \Gamma : S_\lm = 0 \}, 
\qquad 
\Gamma_1 := \{ \lm \in \Gamma : S_\lm > 0 \}.
\end{equation}

It is observed in \cite{Baldi.Haus.JDDE}, \cite{Baldi.Haus.SIAM} that,  
if $(u(t), v(t))$ solves \eqref{3101.16} on some time interval, 
then, for every $\lm \in \Gamma$, 
calculating the Fourier coefficients of $W(u,v)$ in \eqref{W decomp}
and taking the sum over all indices $k \in \Z^d$ on the sphere $|k| = \lm$, 
the corresponding quantities $S_\lm(t), B_\lm(t)$ in \eqref{def.SB} 
satisfies the equations 
\begin{align} 
\pa_t S_\lm 
& = \frac{3 i}{32}  \sum_{\begin{subarray}{c} \a,\b \in \Gamma \\  \a + \b = \lm \end{subarray}} 
( B_\a B_\b \overline{B_\lm} - \overline{B_\a} \overline{B_\b} B_\lm ) 
\a \b \lm
+ \frac{3 i}{16}  \sum_{\begin{subarray}{c} \a,\b \in \Gamma \\  \a - \b = \lm \end{subarray}}
( B_\a \overline{B_\b} \overline{B_\lm} - \overline{B_\a} B_\b B_\lm ) 
\a \b \lm 	
+ R_{S_\lm},
\notag \\ 
\pa_t B_\lm 
& = - 2 i (1 + \mP) \Big( \lm + \frac14 \lm^2 S_\lm \Big) B_\lm 
+ R_{B_\lm}
\label{3105.8}
\end{align}
on the same time interval, 
where the terms $R_{S_\lm}, R_{B_\lm}$ satisfies the following estimates. 
 

\begin{lemma}[Lemma 2.2 of \cite{Baldi.Haus.SIAM}] 
\label{lemma:3005.3}
Let $(u,v) \in H^{m_1}_0(\T^d,c.c.)$ with $\| u \|_{m_1} \leq \d$, 
where $\d$ is the constant given by Lemma \ref{lemma:Lemma.2.9.SIAM}
and appearing in \eqref{twice.norm.sec.prepar}. 
Then, for all $\lm \in \Gamma$, the terms 
$R_{S_\lm}, R_{B_\lm}$ in \eqref{3105.8} satisfy
\begin{equation} \label{3005.3}
|R_{S_\lm}| \leq C \| u \|_{m_1}^6 S_\lm, \quad \ 
|R_{B_\lm}| \leq C \| u \|_{m_1}^4 S_\lm,
\end{equation}
where $C>0$ is a universal constant.
\end{lemma}

Define 
\begin{equation}  \label{def.Z.th}
Z_{\a \b \lm} := B_\a B_\b \ov{B_\lm}, \quad \ 
\th_{\a\b\lm} := \mathrm{Im} (Z_{\a\b\lm}) 
= \frac{B_\a B_\b \overline{B_\lm} - \overline{B_\a B_\b}B_\lm}{2i}.
\end{equation}
If $S_\lm, B_\lm$ satisfy system \eqref{3105.8}, 
then $S_\lm, Z_{\a\b\lm}$ satisfy 
\begin{align} 
\pa_t S_\lm 
&= - \frac{3}{16} 
\sum_{\begin{subarray}{c} \a,\b \in \Gamma \\  \a + \b = \lm \end{subarray}} 
\th_{\a \b \lm} \, \a \b \lm 
+ \frac{3}{8} 
\sum_{\begin{subarray}{c} \a,\b \in \Gamma \\  \b + \lm = \a \end{subarray}} 
\th_{\b \lm \a}  \, \a \b \lm
+ R_{S_\lm}, 
\label{0906.11}
\\
\pa_t Z_{\a\b\lm}
& = -2i (1 + \mP) \Big(\a + \b - \lm +\frac14 (\a^2 S_\a + \b^2 S_\b - \lm^2 S_\lm)\Big) Z_{\a\b\lm}
+ \widetilde R_{Z_{\a\b\lm}},
\notag 
\end{align}
where 
\[ 
\widetilde R_{Z_{\a\b\lm}}
:= R_{B_\a} B_\b \overline{B_\lm} 
+ B_\a R_{B_\b} \overline{B_\lm} 
+ B_\a B_\b \overline{R_{B_\lm}}.
\] 
For $\a + \b = \lm$,
isolating the first nontrivial contribution from terms of higher homogeneity orders, 
one has 
\begin{equation}  \label{3105.11}
\pa_t Z_{\a\b\lm}
= - \frac{i}{2} (\a^2 S_\a + \b^2 S_\b - \lm^2 S_\lm) Z_{\a\b\lm} 
+ R_{Z_{\a\b\lm}}
\end{equation}
where 
\begin{equation}  \label{3105.12}
R_{Z_{\a\b\lm}}
:= - \frac{i}{2} \mP_2 (\a^2 S_\a + \b^2 S_\b - \lm^2 S_\lm) Z_{\a\b\lm} 
+ \widetilde R_{Z_{\a\b\lm}}.
\end{equation}
The remainder $R_{S_\lm}$ in \eqref{0906.11} has been bounded in Lemma \ref{lemma:3005.3}. 
The remainder $R_{Z_{\a\b\lm}}$ in \eqref{3105.12} is estimated in the next lemma.

\begin{lemma}[Lemma 2.5 of \cite{Baldi.Haus.SIAM}] 
\label{lemma:3105.13}
Assume the hypotheses of Lemma \ref{lemma:3005.3}.
Then for all $\a,\b,\lm \in \Gamma$ with $\a+\b=\lm$
the remainder $R_{Z_{\a\b\lm}}$ defined in \eqref{3105.12} satisfies 
\begin{equation} \label{0106.10}
|R_{Z_{\a\b\lm}}| \leq C \| u \|_{m_1}^4 S_\a S_\b S_\lm,
\end{equation}
where $C>0$ is a universal constant.
\end{lemma}

\subsection{The truncated effective system}

If we remove the remainders from equations \eqref{0906.11} and \eqref{3105.11}, 
we obtain a system that we call \textbf{\emph{truncated effective system}},
which is
\begin{align} 
\pa_t S_\lm 
& = - \frac{3}{16} 
\sum_{\begin{subarray}{c} \alpha, \beta \in \Gamma \\ 
\alpha + \beta = \lm \end{subarray}} 
\th_{\alpha \beta \lm} \alpha \beta \lm 
+ \frac{3}{8} 
\sum_{\begin{subarray}{c} \alpha, \beta \in \Gamma \\ 
\beta + \lm = \alpha \end{subarray}} 
\th_{\beta \lm \alpha} \beta \lm \alpha,
\label{pat.S.lm}
\\
\pa_t Z_{\alpha \beta \lm} 
& = - \frac{i}{2} \om_{\alpha \beta \lm} Z_{\alpha \beta \lm},
\label{pat.Z}
\end{align}
where
\begin{equation} \label{def.th.om}
\om_{\alpha \beta \lm}
:= \alpha^2 S_\a + \beta^2 S_\beta - \lm^2 S_\lm.
\end{equation}

For all $\a, \b, \lm \in \Gamma$ such that $\a + \b = \lm$, let 
\[
r_{\a \b \lm} := \Re (Z_{\a \b \lm}).
\]
Since $\om_{\a \b \lm}$ is real, 
the real and imaginary part of equation \eqref{pat.Z} is given by the system
\begin{equation} \label{syst.r.th}
\pa_t r_{\a \b \lm} = \frac12 \om_{\a \b \lm} \th_{\a \b \lm}, 
\qquad 
\pa_t \th_{\a \b \lm} = - \frac12 \om_{\a \b \lm} r_{\a \b \lm}. 
\end{equation}
The solutions $(r_{\a \b \lm}(t) , \th_{\a \b \lm}(t))$ of 
\eqref{syst.r.th} remain on a circle, because they satisfy
\begin{equation} \label{Z.prime.int}
\pa_t ( |Z_{\a \b \lm}|^2 )
= \pa_t \big( r_{\a \b \lm}^2 + \th_{\a \b \lm}^2 \big) = 0.
\end{equation}
Thus $|Z_{\a \b \lm}|$ is a prime integral of the truncated effective system 
\eqref{pat.S.lm}-\eqref{pat.Z}.
Therefore, if $S_\lm, Z_{\a \b \lm}$ solve \eqref{pat.S.lm}-\eqref{pat.Z}, 
then the real and imaginary part of $Z_{\a \b \lm}$ satisfy
\begin{equation} \label{r.th.rho.ph}
r_{\a \b \lm}(t) = \rho_{\a \b \lm} \cos( \ph_{\a \b \lm}(t)),
\qquad 
\th_{\a \b \lm}(t) = \rho_{\a \b \lm} \sin( \ph_{\a \b \lm}(t)),
\end{equation}
where 
\[
\rho_{\a \b \lm} = |Z_{\a \b \lm}| \geq 0
\]
is a constant, 
and $\ph_{\a \b \lm}(t)$ is an angle. 
Moreover, plugging \eqref{r.th.rho.ph} into \eqref{syst.r.th} gives 
the equation for the evolution of the angle.
Hence the truncated effective system \eqref{pat.S.lm}-\eqref{pat.Z} 
becomes 
\begin{align} 
\pa_t S_\lm 
& = - \frac12
\sum_{\begin{subarray}{c} \alpha, \beta \in \Gamma \\ 
\alpha + \beta = \lm \end{subarray}} 
c_{\a \b \lm} \sin(\ph_{\a \b \lm})
+ \sum_{\begin{subarray}{c} \alpha, \beta \in \Gamma \\ 
\beta + \lm = \alpha \end{subarray}} 
c_{\beta \lm \alpha} \sin(\ph_{\beta \lm \alpha}),
\label{pat.S.lm.bis}
\\
\pa_t \ph_{\a \b \lm} 
& = - \frac12 
(\a^2 S_\a + \b^2 S_\b - \lm^2 S_\lm),
\label{pat.ph}
\end{align}
where 
\begin{equation} \label{def.c}
c_{\a \b \lm} := \frac38 \rho_{\alpha \beta \lm} \alpha \beta \lm
= \frac38 |Z_{\a \b \lm}| \a \b \lm 
= \frac38 |B_\a| |B_\b| |B_\lm| \a \b \lm
\end{equation}
is a constant. 
We note that \eqref{pat.S.lm.bis} and \eqref{pat.ph}
form a closed system for the variables $S_\lm, \ph_{\a\b\lm}$. 


\section{The truncated effective system with two triplets}\label{sec:2dofHS}

We consider the case in which the Fourier support $\Gamma_1$
in \eqref{def.Gamma.0.1.general} has only $4$ distinct elements, 
forming two resonant triplets with two elements in common, 
in the following way: 
\begin{gather} 
\Gamma_1 = \{ \a_1, \a_2, \a_3, \a_4 \}, 
\qquad  
\a_1 < \a_2 < \a_3 < \a_4, 
\notag \\
\a_1 + \a_2 = \a_3, 
\qquad 
\a_2 + \a_3 = \a_4, 
\qquad 
2 \a_1 \neq \a_2.
\label{Gamma1.Fib}
\end{gather}
We can assume, without loss of generality, that the four elements of $\Gamma_1$ 
are natural numbers. 
Examples of such sets are any four consecutive elements of the Fibonacci sequence 
greater than $1$, like $\{ 2,3,5,8 \}$, or, more generally, any set of the form
\begin{equation} \label{alpha.1234.Fib}
\a_1 := m, \quad \ 
\a_2 := m+p, \quad \ 
\a_3 := 2m+p, \quad \ 
\a_4 := 3m + 2p,
\end{equation}
where $m,p$ are distinct positive integers with 
\[
2 \leq m < p.
\]

\begin{lemma} \label{lemma:no.other.triplets}
Assume \eqref{Gamma1.Fib}, \eqref{alpha.1234.Fib}. 
If $\a, \b, \lm \in \Gamma_1$ satisfy 
$\a + \b = \lm$, then the ordered triplet $(\a, \b, \lm)$ must be 
\[
(\a_1, \a_2, \a_3) 
\quad \text{or} \quad 
(\a_2, \a_1, \a_3) 
\quad \text{or} \quad 
(\a_2, \a_3, \a_4) 
\quad \text{or} \quad 
(\a_3, \a_2, \a_4),
\]
and there are no other options. 
\end{lemma}

\begin{proof}
For example, one has 
\[
\a_3 = \a_1 + \a_2 
< \a_2 + \a_2 
< \a_2 + \a_3 = \a_4,
\]
therefore $2 \a_2 \notin \Gamma_1$, 
and the triplet $(\a_2, \a_2, 2 \a_2)$ is not admissible; 
the other cases can be checked similarly.
\end{proof}

To slightly shorten the notation, we denote 
\[
S_1 := S_{\a_1}, \quad \  
\ph_{123} := \ph_{\a_1 \a_2 \a_3}, \quad \  
c_{123} := c_{\a_1 \a_2 \a_3},
\] 
and so on.
Hence, system \eqref{pat.S.lm.bis}-\eqref{pat.ph} becomes
\begin{equation}\label{syst.6.eq}
\begin{aligned} 
\pa_t S_1 
& = c_{123} \sin(\ph_{123}),
\\
\pa_t S_2 
& = c_{123} \sin(\ph_{123}) 
+ c_{234} \sin(\ph_{234}),
 \\
\pa_t S_3
& = - c_{123} \sin(\ph_{123}) 
+ c_{234} \sin(\ph_{234}),
 \\
\pa_t S_4
& = - c_{234} \sin(\ph_{234}),
 \\
\pa_t \ph_{123} 
& = - \frac12 (\a_1^2 S_1 + \a_2^2 S_2 - \a_3^2 S_3),
 \\
\pa_t \ph_{234} 
& = - \frac12 (\a_2^2 S_2 + \a_3^2 S_3 - \a_4^2 S_4),
\end{aligned}
\end{equation}
which is a system of 6 equations in 6 unknowns.

\subsection{Meaningfulness condition for the solutions} 

Our strategy is this: 
We want to find solutions of the truncated effective system \eqref{syst.6.eq}
with a prescribed, interesting dynamical behavior, 
and to show that the solution of the effective system 
\eqref{0906.11}-\eqref{3105.11} is so close 
to the solution of the truncated effective system \eqref{syst.6.eq}
that the dynamical behaviors of the two solutions are very similar, 
on a sufficiently long interval of time. Later, in Section \ref{sec:gronwall}, we show that there exists a solution of the original PDE \eqref{p1} which is very close (up to the change of coordinates $\Phi$) to the solution of the effective system.


Hence, we look for solutions of the truncated effective system \eqref{syst.6.eq}
that satisfy the natural meaningfulness condition 
required by system \eqref{0906.11}-\eqref{3105.11}, 
which is simply this: 
If a solution is defined on a time interval $[0,T]$, 
then it must satisfy
\begin{equation} \label{nec.cond.S}
S_n(t) > 0 \quad \ 
\forall n = 1,2,3,4, \quad \ 
\forall t \in [0,T].
\end{equation}
So, we reject any solution of \eqref{syst.6.eq} such that 
some of the $S_n$ becomes non-positive at some time $t$
(recall definitions \eqref{def.SB} and \eqref{def.Gamma.0.1.general}).  

In the following analysis, we first ignore the constrain \eqref{nec.cond.S}; 
later, we will select only solutions satisfying it. 
Analogously, we first consider the coefficients $c_{123}, c_{234}$ 
as any two given constants; 
later, we will go back to the identities \eqref{def.c}.

\subsection{First integrals and a linear change of coordinates}
Given any linear combination $E := \mu_1 S_1 + \mu_2 S_2 + \mu_3 S_3 + \mu_4 S_4$ of $S_1, \ldots, S_4$ 
with constant real coefficients $\mu_1, \ldots, \mu_4$, 
we have
\[
\pa_t E = 
(\mu_1 + \mu_2 - \mu_3) c_{123} \sin(\ph_{123}) 
+ (\mu_2 + \mu_3 - \mu_4) c_{234} \sin(\ph_{234})
\]
along the solutions of system \eqref{syst.6.eq}. 
Hence any $E$ with coefficients $\mu_1, \ldots, \mu_4$ satisfying 
\[
\mu_1 + \mu_2 - \mu_3 = 0, 
\qquad 
\mu_2 + \mu_3 - \mu_4 = 0
\]
is a first integral. 
%
We choose the two functionally independent first integrals
\begin{equation} \label{def.E1.E2}
E_1 := S_1 + S_3 + S_4, \quad \ 
E_2 := S_2 + S_3 + 2 S_4.
\end{equation}

\begin{remark} 
\label{rem:norm.12.prime.int}
One has 
\begin{equation} \label{norm.12.prime.int}
\a_1 E_1 + \a_2 E_2 
= \a_1 (S_1 + S_3 + S_4)
+ \a_2 (S_2 + S_3 + 2 S_4)
= \sum_{n=1}^4 \a_n S_n 
\end{equation}
because 
$\a_1 + \a_2 = \a_3$ 
and 
$\a_1 + 2 \a_2 
= \a_2 + (\a_1 + \a_2) 
= \a_2 + \a_3 
= \a_4$. 
Hence, when $S_n$ are given by \eqref{def.SB}, identity \eqref{norm.12.prime.int} 
implies that the Sobolev norm $\| u \|_{\frac12}^2 = \sum_{n=1}^4 \a_n S_n$ 
is also a first integral of \eqref{syst.6.eq}. 
\end{remark}


At each time $t$, the values $S_1(t), S_2(t)$ can be obtained 
from $E_1, E_2, S_3(t), S_4(t)$ by \eqref{def.E1.E2}, i.e.,
\begin{equation} \label{S12.from.E12.S34}
S_1(t) = E_1 - S_3(t) - S_4(t), \quad \ 
S_2(t) = E_2 - S_3(t) - 2 S_4(t).
\end{equation}
Hence system \eqref{syst.6.eq} can be reduced to a system of 4 equations 
in the 4 unknowns $S_3, S_4, \ph_{123}, \ph_{234}$, 
obtained by replacing $S_1, S_2$ by \eqref{S12.from.E12.S34} 
in the last two equations of \eqref{syst.6.eq}. 
We get 
\begin{align} 
\pa_t S_3
& = - c_{123} \sin(\ph_{123}) 
+ c_{234} \sin(\ph_{234}),
\notag \\
\pa_t S_4
& = - c_{234} \sin(\ph_{234}),
\notag \\
\pa_t \ph_{123} 
& = - \frac12 ( \a_1^2 E_1 + \a_2^2 E_2 )
+ \frac12 (\a_1^2 + \a_2^2 + \a_3^2) S_3 
+ \frac12 (\a_1^2 + 2 \a_2^2) S_4,
\notag \\
\pa_t \ph_{234} 
& = - \frac12 \a_2^2 E_2 
- \frac12 (\a_3^2 - \a_2^2) S_3 
+\frac12 (2 \a_2^2 + \a_4^2) S_4.
\label{syst.4.eq}
\end{align}

We summarize the observations above in the following lemma.

\begin{lemma} \label{lemma:change.6.eq.4.eq}
Let $c_{123}, c_{234}$ be any two constants. 
The following properties hold.
\begin{itemize}
	\item[(i)] Let $(S_1(t), S_2(t), S_3(t), S_4(t), \ph_{123}(t), \ph_{234}(t))$ 
be a solution of system \eqref{syst.6.eq} on some time interval $I$.    
Then $E_1, E_2$ defined by \eqref{def.E1.E2} are constant in time and $(S_3(t), S_4(t), \ph_{123}(t)$, $\ph_{234}(t))$ 
solves system \eqref{syst.4.eq} on $I$. 

\item[(ii)] Let $E_1, E_2$ be constants, 
and let $(S_3(t), S_4(t), \ph_{123}(t), \ph_{234}(t))$ be a solution 
of system \eqref{syst.4.eq} on some time interval $I$. 
Define the functions $S_1(t), S_2(t)$ by the identities \eqref{S12.from.E12.S34}. 
Then $(S_1(t), S_2(t), S_3(t), S_4(t), \ph_{123}(t), \ph_{234}(t))$ 
solves system \eqref{syst.6.eq} on $I$. 
\end{itemize}
\end{lemma}



We note that the sum of the first two equations in \eqref{syst.4.eq}
does not contain the angle $\ph_{234}$. 
Hence, we consider a linear change of variable 
that treats the sum $S_3 + S_4$ as a new variable. 




\begin{lemma} \label{lemma:change.4.eq.xy}
Let $c_{123}, c_{234}$ be any two constants. 
The following properties hold.
\begin{itemize}
\item[(i)] Let $E_1, E_2$ be constants, 
and let $(S_3(t), S_4(t), \ph_{123}(t), \ph_{234}(t))$ be a solution 
of system \eqref{syst.4.eq} on some time interval $I$. 
Define the functions $x_1(t), x_2(t), y_1(t), y_2(t)$ by the change of coordinates
\begin{equation} \label{change.xy}
	\ph_{123} = x_1, \quad \ 
	\ph_{234} = x_2, \quad \ 
	S_3 = y_1 - y_2, \quad \ 
	S_4 = y_2.
\end{equation}
%
Then $(x_1(t), x_2(t), y_1(t), y_2(t))$ solves 
\begin{equation}\label{syst.xy}
\begin{aligned} 
	\pa_t x_1 
	& = - \frac12 b_1 
	+ \frac12 (\a_1^2 + \a_2^2 + \a_3^2) y_1 
	- \frac12 (\a_3^2 - \a_2^2) y_2,\\
		\pa_t x_2 
	& = - \frac12 b_2 
	- \frac12 (\a_3^2 - \a_2^2) y_1
	+ \frac12 (\a_2^2 + \a_3^2 + \a_4^2) y_2\\
		\pa_t y_1 
	& = - c_{123} \sin(x_1),
 \\
	\pa_t y_2
	& = - c_{234} \sin(x_2),
\end{aligned}
\end{equation}
on $I$, where $b_1, b_2$  are the constants
\begin{equation} \label{def.b12}
	b_1 := \a_1^2 E_1 + \a_2^2 E_2,
	\quad \ 
	b_2 := \a_2^2 E_2.
\end{equation}
\item[(ii)] Let $b_1, b_2$ be constants, and let 
$(x_1(t), x_2(t), y_1(t), y_2(t))$ be a solution 
of system \eqref{syst.xy} on some time interval $I$.
Define the constants $E_1, E_2$ as
\begin{equation} \label{E12.from.b12}
E_1 = \frac{b_1 - b_2}{\a_1^2}, \qquad 
E_2 = \frac{b_2}{\a_2^2},
\end{equation}
and define the functions $S_3(t), S_4(t), \ph_{123}(t), \ph_{234}(t)$ 
by \eqref{change.xy}. 
Then $(S_3(t), S_4(t), \ph_{123}(t)$, $\ph_{234}(t))$ 
solves system
\eqref{syst.4.eq} 
on $I$.
\end{itemize}
\end{lemma}

\subsection{The Hamiltonian structure}

System \eqref{syst.xy} is the 2-dimensional Hamiltonian system 
\[
\dot x_n = \pa_{y_n} H(x,y), \quad \ 
\dot y_n = - \pa_{x_n} H(x,y), \quad \ 
n = 1,2,
\]
with Hamiltonian 
\begin{equation}\label{def.H}
H(x,y) = - c_{123} \cos(x_1) - c_{234} \cos(x_2) 
- \frac12 b \cdot y + \frac14 Ay \cdot y
\end{equation}
where $b=(b_1, b_2)\in\R^2$ 
and $A$ is the matrix 
\begin{equation} \label{def.matrix.A}
A := \begin{pmatrix} (\a_1^2 + \a_2^2 + \a_3^2) \ & \ - (\a_3^2 - \a_2^2) 
\\
- (\a_3^2 - \a_2^2) \ & \ (\a_2^2 + \a_3^2 + \a_4^2) 
\end{pmatrix}.
\end{equation}

%
%

The matrix $A$ is symmetric, positive definite and  invertible with 
\begin{equation} \label{inv.matrix.A}
	A^{-1} = \frac{1}{\det A} 
	\begin{pmatrix} (\a_2^2 + \a_3^2 + \a_4^2)  \ & \ (\a_3^2 - \a_2^2) 
		\\
		(\a_3^2 - \a_2^2) \ & \ (\a_1^2 + \a_2^2 + \a_3^2)
	\end{pmatrix}
\end{equation}
and
\begin{equation}\label{det.A}
\det A 
 = \a_1^2 \a_4^2 + (\a_1^2 + \a_4^2)(\a_2^2 + \a_3^2) + 4 \a_2^2 \a_3^2 
> 0.
\end{equation}
The invertibility of $A$ is the so-called \emph{twist condition} for the Hamiltonian $H$;
thanks to it, we can eliminate the linear term $b \cdot y$ from the Hamiltonian 
by a translation of the $y$ variables
\begin{equation} \label{change.xy.tilde}
x_1 = \tilde x_1, \quad \ 
x_2 = \tilde x_2, \quad \ 
y_1 = q_1 + \tilde y_1, \quad \ 
y_2 = q_2 + \tilde y_2, 
\end{equation}
with $q = A^{-1} b$, namely 
\begin{equation} \label{formula.q12}
	q_1 = \frac{\a_2^2 + \a_3^2 + \a_4^2}{\det A} \, b_1 
	+ \frac{\a_3^2 - \a_2^2}{\det A} \, b_2, 
	\qquad 
	q_2 = \frac{\a_3^2 - \a_2^2}{\det A} \, b_1 
	+ \frac{\a_1^2 + \a_2^2 + \a_3^2}{\det A} \, b_2.
\end{equation}
This change of coordinates is symplectic and the new Hamiltonian is just $\tilde H(\tilde x, \tilde y) 
= H(\tilde x, q + \tilde y)$, whose equations are given by
\begin{equation} \label{syst.xy.tilde}
	\begin{aligned} 
\pa_t \tilde x_1 
& = \tfrac12 (\a_1^2 + \a_2^2 + \a_3^2) \tilde y_1 
- \tfrac12 (\a_3^2 - \a_2^2) \tilde y_2,
 \\
\pa_t \tilde y_1 
& = - c_{123} \sin(\tilde x_1),
\\
\pa_t \tilde x_2 
& = - \tfrac12 (\a_3^2 - \a_2^2) \tilde y_1
+ \tfrac12 (\a_2^2 + \a_3^2 + \a_4^2) \tilde y_2,
 \\ 
\pa_t \tilde y_2
& = - c_{234} \sin(\tilde x_2).
\end{aligned}
\end{equation}
%
%
The equivalence of systems \eqref{syst.xy} and \eqref{syst.xy.tilde} 
is described in the following lemma.

\begin{lemma} \label{lemma:change.xy.xy.tilde}
Let $c_{123}, c_{234}$ be any two constants. 
The following properties hold.
\begin{itemize}
\item[(i)] Let $b_1, b_2$ be constants, and let 
$(x_1(t), x_2(t), y_1(t), y_2(t))$ be a solution 
of system \eqref{syst.xy} on some time interval $I$.
Define the constants $q_1, q_2$ by \eqref{formula.q12}, 
and define the functions $\tilde x_1(t), \tilde x_2(t), \tilde y_1(t), \tilde y_2(t)$ 
by \eqref{change.xy.tilde}.
Then $(\tilde x_1(t), \tilde x_2(t), \tilde y_1(t), \tilde y_2(t))$ 
solves \eqref{syst.xy.tilde}. 

\item[(ii)] Let $(\tilde x_1(t), \tilde x_2(t), \tilde y_1(t), \tilde y_2(t))$ 
be a solution of system \eqref{syst.xy.tilde} on some  interval $I$. 
Let $q_1, q_2$ be any two real numbers. 
Define  constants $b_1, b_2$ by the identity $b = Aq$, i.e., define 
\begin{equation} \label{b12.from.q12} 
b_1 = (\a_1^2 + \a_2^2 + \a_3^2) q_1 - (\a_3^2 - \a_2^2) q_2, 
\qquad 
b_2 = - (\a_3^2 - \a_2^2) q_1 + (\a_2^2 + \a_3^2 + \a_4^2) q_2,
\end{equation}
and define the functions $x_1(t), x_2(t), y_1(t), y_2(t)$ by \eqref{change.xy.tilde}. 
Then $(x_1(t), x_2(t), y_1(t), y_2(t))$ solves \eqref{syst.xy} on $I$.
\end{itemize}
\end{lemma}

\subsection{Normalization of coefficients by rescaling}


Now we want to normalize
the leading coefficients of system \eqref{syst.xy.tilde}, 
using a rescaling of the time variable 
and dilations of the $\tilde y$ variables. 
We consider the change of variables
\begin{equation} \label{change.xi.eta}
\tilde x_1(t) = \xi_1(\tB t), \quad \ 
\tilde x_2(t) = \xi_2(\tB t), \quad \ 
\tilde y_1(t) = \tA_1 \eta_1(\tB t), \quad \ 
\tilde y_2(t) = \tA_2 \eta_2(\tB t), \quad 
\end{equation}
where $\tA_1, \tA_2, \tB$ are defined as follows.
We assume that 
\begin{equation} \label{nec.cond.c.123}
c_{123} > 0
\end{equation}
and we fix 
\begin{equation} \label{fix.A1.B}
\tA_1 := \Big( \frac{2 c_{123}}{\a_1^2 + \a_2^2 + \a_3^2} \Big)^{\frac12}, 
\qquad 
\tB := \Big( \frac{c_{123}(\a_1^2 + \a_2^2 + \a_3^2)}{2} \Big)^{\frac12},
\end{equation}
%
%
%
\begin{equation} \label{fix.A2}
\tA_2 := \frac{2 \mathtt{B}}{\a_2^2 + \a_3^2 + \a_4^2}.
\end{equation}
Note that $\tA_1$ and $\tA_2$ are related by 
\begin{equation} \label{def.gamma}
\tA_1 = \tA_2 \g, \qquad \text{with}\qquad
\gamma := \frac{\a_2^2 + \a_3^2 + \a_4^2}{\a_1^2 + \a_2^2 + \a_3^2}.
\end{equation}
Thus, system \eqref{syst.xy.tilde} becomes 
\begin{equation}\label{syst.xi.eta.partially.normalized} 
\begin{split}
\dot \xi_1 & = \eta_1 - \mu_1 \eta_2,
\\
\dot \eta_1 & = - \sin(\xi_1), 
 \\
\dot \xi_2 & = \eta_2 - \mu_2 \eta_1, 
 \\
\dot \eta_2 & = - \lm  \sin(\xi_2),
\end{split}
\end{equation}
where 
\begin{align}
\label{def.mu.1}
\mu_1 & := \frac{\a_3^2 - \a_2^2}{\a_2^2 + \a_3^2 + \a_4^2}
= \frac{(\a_3^2 - \a_2^2) \tA_2}{2\tB},
\\ 
\mu_2 & := \frac{\a_3^2 - \a_2^2}{\a_1^2 + \a_2^2 + \a_3^2}
= \mu_1 \gamma
= \frac{(\a_3^2 - \a_2^2) \tA_1}{2\tB},
\label{def.mu.2}\\
\lm&:= \frac{c_{234}(\a_2^2 + \a_3^2 + \a_4^2)}{c_{123}(\a_1^2 + \a_2^2 + \a_3^2)}
	= \frac{c_{234}}{c_{123}}\,\gamma
	= \frac{c_{234}}{\tA_2 \tB} .
\label{def.lambda}
\end{align}

We observe that the system with $\mu_1=0$ is given by the sum of two uncoupled Hamiltonians, while for $\mu_1\neq 0$ the Hamiltonian structure is lost.
The equivalence of systems \eqref{syst.xy.tilde} 
and \eqref{syst.xi.eta.partially.normalized} is described in the following lemma.

\begin{lemma} \label{lemma:change.xy.tilde.xi.eta}
The following statements are satisfied.
\begin{itemize}
\item[(i)] Let $c_{123}, c_{234}$ be any two constants, with $c_{123} > 0$,
and let $(\tilde x_1(t), \tilde x_2(t), \tilde y_1(t), \tilde y_2(t))$ 
be a solution of system \eqref{syst.xy.tilde} on some time interval $[0,T]$. 
Let $\tA_1, \tA_2, \tB$ be the constants defined in \eqref{fix.A1.B}, \eqref{fix.A2}. 
Define the functions $\xi_1, \xi_2, \eta_1, \eta_2$ as 
\begin{equation} \label{xi.eta.from.xy.tilde}
\xi_1(t) = \tilde x_1 \Big( \frac{t}{\tB} \Big), \quad \ 
\xi_2(t) = \tilde x_2 \Big( \frac{t}{\tB} \Big), \quad \ 
\eta_1(t) = \frac{1}{\tA_1} \tilde y_1 \Big( \frac{t}{\tB} \Big), \quad \ 
\eta_2(t) = \frac{1}{\tA_2} \tilde y_2 \Big( \frac{t}{\tB} \Big). \quad 
\end{equation}
Define the constants $\mu_1, \mu_2, \lm$ 
by \eqref{def.mu.1}, \eqref{def.mu.2}, \eqref{def.lambda}.
Then $(\xi_1(t), \xi_2(t), \eta_1(t), \eta_2(t))$ 
solves \eqref{syst.xi.eta.partially.normalized} on the time interval $[0, \mathtt{B} T]$. 

\item[(ii)] Let $\lm$ be any constant. 
Let $\mu_1, \mu_2$ be the constants defined in \eqref{def.mu.1}, \eqref{def.mu.2}. 
Let $(\xi_1(t), \xi_2(t), \eta_1(t), \eta_2(t))$ 
be a solution of system \eqref{syst.xi.eta.partially.normalized} 
on some time interval $[0, T]$. 
Let $\tA_1$ be any positive constant. 
Define the constant $c_{123}$ as 
\begin{equation} \label{c123.from.A1}
c_{123} = \frac{\a_1^2 + \a_2^2 + \a_3^2}{2} \tA_1^2
\end{equation}
and define the constant $c_{234}$ by means of \eqref{def.lambda}, 
i.e., $c_{234} = c_{123} \lm / \gamma$, where $\gamma$ is defined in \eqref{def.gamma}.
Define the constant $\mathtt{B}$ by the second identity in \eqref{fix.A1.B},
and the constant $\mathtt{A}_2$ by \eqref{fix.A2}.
Define the functions  
$\tilde x_1, \tilde x_2, \tilde y_1, \tilde y_2$ 
by \eqref{change.xi.eta}. 
Then $(\tilde x_1(t), \tilde x_2(t), \tilde y_1(t), \tilde y_2(t))$ 
solve \eqref{syst.xy.tilde} on the time interval $[0, \mathtt{B}^{-1} T]$. 
\end{itemize}
\end{lemma}

\subsection{Large Fourier frequency as a perturbation parameter}

Recall the definition \eqref{alpha.1234.Fib} of $\a_1, \ldots, \a_4$ 
as functions of the two integer parameters $m,p$.  
Then 
\begin{equation}\label{coeff.matrix.A}
\begin{aligned}
\a_3^2 - \a_2^2 
& = (2m+p)^2 - (m+p)^2 
= 3m^2 + 2mp,
\notag \\
\a_1^2 + \a_2^2 + \a_3^2
& = m^2 + (m+p)^2 + (2m+p)^2
= 6m^2 + 6mp + 2 p^2,
\notag \\
\a_2^2 + \a_3^2 + \a_4^2
& = (m+p)^2 + (2m+p)^2 + (3m+2p)^2
= 14 m^2 + 18 mp + 6 p^2.
\end{aligned}
\end{equation}
We note that the monomial $p^2$ cancels out in the difference $\a_3^2 - \a_2^2$, 
while it is present in the other two sums. 
For this reason, taking $p$ large with respect to $m$ 
gives a small parameter, which we will use in our perturbation analysis
(the other small parameter of the problem is the size of the solution, 
i.e., the size of the initial data of the Kirchhoff equation). 
Denoting
\begin{equation} \label{def.sigma}
\sigma := \frac{m}{p},
\end{equation}
one has
\begin{align}
\mu_1 
& = \frac{2 \sigma + 3 \sigma^2}{6 + 18 \sigma + 14 \sigma^2}
= \sigma \Big( \frac13 + \tilde{\mu}_1 (\sigma) \Big),
\qquad 
\g 
= \frac{6 + 18 \sigma + 14 \s^2}{2 + 6 \s + 6 \s^2} 
= 3 + O(\s), 
\notag \\
\mu_2 
& = \mu_1 \g 
= \frac{2 \sigma + 3 \sigma^2}{2 + 6 \s + 6 \s^2} 
= \s  \big( 1 + \tilde{\mu}_2 (\sigma) \big),
\qquad \tilde{\mu}_1 (\sigma), \tilde{\mu}_2 (\sigma) = O (\sigma) 
\quad \text{as } \s \to 0,
\label{Taylor.mu.12.gamma}
\end{align}
where $\tilde \mu_1(\s), \tilde \mu_2(\s)$ are defined 
by the identities \eqref{Taylor.mu.12.gamma}.
We also note that $1 < \g < 3$ for all $\s > 0$.

Thus, for $\s = m/p$ small, 
the ``coupling'' terms $\mu_1 \eta_2$ and $\mu_2 \eta_1$ 
in system \eqref{syst.xi.eta.partially.normalized}
can be considered as perturbations of the 
``unperturbed'' system of two uncoupled pendulums
\begin{equation} \label{unpert.syst.lambda}
\begin{cases}
\dot \xi_1 = \eta_1,
\\
\dot \eta_1 = - \sin(\xi_1), 
\end{cases}
\qquad \quad 
\begin{cases}
\dot \xi_2 = \eta_2, 
\\
\dot \eta_2 = - \lm  \sin(\xi_2).
\end{cases}
\end{equation}


We want to normalize also the coefficient $\lm$ appearing in the last equation 
of system \eqref{syst.xi.eta.partially.normalized}. 
Later, we will see that this normalization corresponds 
to a constraint on the initial data for equation \eqref{pat.Z}; 
at this stage, however, we simply observe that the parameter $\lm$ 
in part $(ii)$ of Lemma \ref{lemma:change.xy.tilde.xi.eta} 
is not subject to any constraint. 
Thus, in the following analysis we fix $\lm = 1$  
and simply do not consider other values of that parameter. 
For $\lm=1$, system \eqref{syst.xi.eta.partially.normalized} becomes 
\begin{equation}\label{syst.xi.eta.fully.normalized}
\begin{cases}
\dot \xi_1  = \eta_1 - \mu_1 \eta_2,\\
\dot \eta_1  = - \sin(\xi_1), \\
\dot \xi_2  = \eta_2 - \mu_2 \eta_1, \\
\dot \eta_2  = - \sin(\xi_2).
\end{cases}
\end{equation}
System \eqref{syst.xi.eta.fully.normalized} has a conserved quantity (obtained expressing the old Hamiltonian in terms of the new variables), which is
\begin{equation}\label{def.mE}
\mu_2 \Big( \frac12 \eta_1^2  + 1 - \cos \xi_1 \Big) + \mu_1 \Big( \frac12 \eta_2^2  + 1 - \cos \xi_2 \Big) - \mu_1 \mu_2 \eta_1 \eta_2.
\end{equation}
The only parameters in system \eqref{syst.xi.eta.fully.normalized} 
are the constants $\mu_1, \mu_2$ defined in \eqref{def.mu.1}, \eqref{def.mu.2}, 
which depend only on the ratio $\s = m/p$, 
and tend to zero as $\s \to 0$ (see \eqref{Taylor.mu.12.gamma}).
The ``unperturbed'' part of system \eqref{syst.xi.eta.fully.normalized} 
is \eqref{unpert.syst.lambda} with $\lm = 1$, that is, 
\begin{equation} \label{unpert.syst}
\begin{cases}
\dot \xi_1 = \eta_1,
\\
\dot \eta_1 = - \sin(\xi_1), 
\end{cases}
\qquad \quad 
\begin{cases}
\dot \xi_2 = \eta_2, 
\\
\dot \eta_2 = - \sin(\xi_2).
\end{cases}
\end{equation}
System \eqref{unpert.syst} is fully normalized 
and it 
is the 2-dimensional Hamiltonian system 
of two uncoupled normalized pendulums
\[
\dot \xi_n = \pa_{\eta_n} H, \quad \ 
\dot \eta_n = - \pa_{\xi_n} H, \quad \ 
n = 1,2,
\]
with Hamiltonian 
\begin{equation} \label{def.H1.H2}
H = H_1(\xi_1, \eta_1) + H_2(\xi_2,\eta_2),
\qquad
H_n(\xi_n,\eta_n) = \frac12 \eta_n^2+\big(1 - \cos(\xi_n)\big) , 
\quad \ n=1,2.
\end{equation}
The constant term 1 in the formula of $H_n$ has been added 
just to give zero energy to the elliptic equilibrium.


%
%
%
 
It is convenient to consider the system \eqref{syst.xi.eta.fully.normalized} as a perturbed double-pendulum system where the perturbative parameter is given by $\sigma$, instead of $\mu_1, \mu_2$. Namely, using \eqref{Taylor.mu.12.gamma}, 
\begin{equation}\label{syst.xi.eta.fully.normalized2} 
\begin{aligned}
\dot \xi_1 & = \eta_1 - \sigma \Big( \frac13 + \tilde{\mu}_1 (\sigma) \Big) \eta_2, \\
\dot \eta_1 & = - \sin(\xi_1),  
\end{aligned}\qquad \quad
\begin{aligned}
\dot \xi_2 & = \eta_2 - \sigma \big( 1 + \tilde{\mu}_2 (\sigma) \big) \eta_1,  \\
\dot \eta_2 & = - \sin(\xi_2).
\end{aligned}
\end{equation}
The conserved quantity in \eqref{def.mE}, divided by $\sigma$, is
\begin{equation}\label{mE.sigma}
\begin{split}
\mE(\xi_1,\eta_1,\xi_2,\eta_2) :=  & \Big( \frac13 + \tilde{\mu}_1 (\sigma) \Big) H_1(\xi_1,\eta_1) + \big( 1 + \tilde{\mu}_2 (\sigma) \big) H_2(\xi_2,\eta_2)\\
& - \sigma \Big( \frac13 + \tilde{\mu}_1 (\sigma) \Big) \big( 1 + \tilde{\mu}_2 (\sigma) \big) \eta_1 \eta_2.
\end{split}
\end{equation}

\section{Chaos for two weakly coupled pendulums}
\label{sec:chaos.2.pendulums}

In this section we prove the following result about chaotic solutions 
of system \eqref{syst.xi.eta.fully.normalized2}. 
We denote $\N_0 := \{ 0,1,2,\ldots \}$ the set of nonnegative integers.
Given an energy parameter $a\in(0,2)$, we denote by $(\xi_1^*(t), \eta_1^*(t))$ the periodic solution of the pendulum satisfying
\begin{equation}\label{def:periodic_a}
H_1(\xi_1^*(t), \eta_1^*(t)) = a, 
\quad \ 
\xi_1^*(0)=0
\quad \ 
\eta_1^*(0)>0.
\end{equation}
To emphasize its dependence on $a$ we will also denote $\xi_1^*(t)=\xi_1^*(t;a)$ and $\eta_1^*(t)=\eta_1^*(t;a)$.

\begin{proposition}  \label{prop:chaos.xi.eta}
There exists a universal constant $a_0\in(0,2)$ (see Lemma \ref{lemma:a0}) such that the following holds.
Let $a\in [a_0, 2)$ and let $T_a$ be the period of $(\xi_1^*(t;a), \eta_1^*(t;a))$.
There exist universal constants $\s_0 \in (0,1)$, 
$C_1, C_2 > 0$ 
such that for every $\s \in (0, \s_0)$ 
there exists $M_0(\s)$ in the interval
\[
C_1 \log(\s^{-1}) \leq M_0(\s) 
\leq C_2 \log(\s^{-1})
\]
such that the following properties hold. 
Let $(m_0, m_1, m_2, \ldots) = (m_j)_{j \in \N_0}$ 
be any sequence of integers with 
$m_j \geq M_0(\s)$ for all $j \in \N$. 
Then, there exists a solution $(\xi_1, \eta_1, \xi_2, \eta_2)(t)$ 
of system \eqref{syst.xi.eta.fully.normalized2} such that 
the following holds.
\begin{itemize}
\item[$(i)$]
The function $(\xi_1, \eta_1)$ satisfies 
\[
\sup_{t \in \R} |\xi_1(t) - \xi_1^*(t)| \leq C \s, 
\qquad 
\sup_{t \in \R} |\eta_1(t) - \eta_1^*(t)| \leq C \s,
\]
for some universal constant $C>0$.

\item[$(ii)$]
There exists a sequence of times $(t_0, t_1, t_2, \ldots)= (t_j)_{j \in \N_0}$ with 
\begin{equation}\label{def:tj}
t_0 = 0, 
\quad \ 
t_{j+1} = t_j + T_a (m_j + \theta_j),
\quad \ 
0 \leq \theta_j < 1,
\quad \ \forall j \in \N_0
\end{equation}
such that 
\[
{\eta_2(t_j)}=1 
\quad \ \forall j \in \N_0.
\]
Moreover, there exists another sequence of times 
$( \bar t_j )_{j\in \N_0}$ 
satisfying $t_j < \bar t_j < t_{j+1}$ such that
\begin{equation}\label{borraccia}
1 < {\eta_2(t)} \leq 2+C\s \quad \forall t \in (t_j, \bar t_j),
\qquad \quad 
-C\s \leq {\eta_2(t)} < 1 \quad \forall t \in (\bar t_j, t_{j+1})
\end{equation}
and
\begin{equation}  \label{bounds.eta2.in.section.6}
\max_{t \in [t_j, \bar t_j]} \eta_2(t) \geq 2-C\s, 
\qquad 
\min_{t \in [\bar t_j, t_{j+1}]} \eta_2(t) \leq C\s
\end{equation}
for some universal constant $C>0$.
\end{itemize}
\end{proposition}

We remark that $\eta_2=1$ is the value around which $\eta_2(t)$ is oscillating up and down 
with the randomly chosen sequence of times.

In order to prove Proposition \ref{prop:chaos.xi.eta} we shall find a partially hyperbolic periodic orbit of the full system \eqref{syst.xi.eta.fully.normalized2} and show that, for $\sigma>0$ small enough, its stable and unstable invariant manifolds intersect transversally. This will imply the existence of a Smale horseshoe and the existence of symbolic chaotic dynamics.

\begin{remark} \label{rem:fix.a.0}
The pendulum energy $a$ in \eqref{def:periodic_a} 
is used as a free parameter only in this section. 
Proposition \ref{prop:chaos.xi.eta} will be applied  
in Sections \ref{sec:back} and \ref{sec:gronwall} only for $a = a_0$.
Since $a_0$ is a universal constant, 
for $a=a_0$ any quantity depending only on $a$ becomes a universal constant. 
\end{remark}

\subsection{Partially hyperbolic periodic orbit}
\label{subsec:periodic.orbit}

We start by searching for the partially hyperbolic periodic solution. 
The unperturbed system ($\sigma=0$) has plenty of partially hyperbolic periodic orbits, 
which are given for instance by the product of librations in the first pendulum 
(in the plane $(\xi_1, \eta_1)$) 
and the saddle of the second pendulum 
(in the plane $(\xi_2, \eta_2)$). 
We select one of these orbits and we apply an implicit function theorem argument to prove the existence of a nearby periodic orbit with the same period.

We consider the periodic orbit 
\begin{equation} \label{def.mP.t.a}
\mP=\mP(t;a) 
:= (\xi_1^*(t;a), \eta_1^*(t;a), \pi, 0)
\end{equation}
of the unperturbed system \eqref{unpert.syst}
(which is system \eqref{syst.xi.eta.fully.normalized2} with $\sigma=0$), 
where 
$(\xi_1^*, \eta_1^*)$ is defined in \eqref{def:periodic_a} and
$a\in(0,2)$ will be fixed at the end of subsection \ref{subsec:Melnikov.int}.
Note that the elliptic equilibrium 
$(\xi_1, \eta_1) = (0,0)$ of the first pendulum 
has energy $H_1(0,0) = 0$,
and its saddle $(\xi_1, \eta_1) = (\pi,0)$, 
as well as its homoclinic orbits, 
has energy $H_1(\pi,0) = 2$.
The solution $\mP$ is supported on the curve
\begin{equation}\label{def:T0}
\T_0 := \{ (\xi_1, \eta_1, \xi_2, \eta_2) : 
H_1(\xi_1, \eta_1) = a, \ \xi_2 = \pi, \ \eta_2 = 0 \}.
\end{equation}
We denote by $T$ the period $T_a$ of the orbit $\mP$ 
and by $\om := 2\pi / T$ its frequency.
Of course $\T_0, T, \om$ depend on the energy parameter $a$; 
in fact, all the quantities in the present subsection and in the next one 
(included, in particular, the smallness radius $\s_1$ given by Proposition \ref{prop:hyppo})
depend on $a$. 
Nonetheless, in general, we do not indicate explicitly the dependence on $a$; 
we just underline that, after fixing $a$, 
every quantity appearing in subsections 
\ref{subsec:periodic.orbit} and 
\ref{subsec:Melnikov.int}
will be determined, 
with no dependence on any other hidden parameter.


The unperturbed ($\sigma=0$) homoclinic manifold of $\mathcal{P}$ is 
\begin{equation}\label{unpinvman}
W_0(\mathcal{P}) = \{ (\xi_1, \eta_1, \xi_2, \eta_2) : 
H_1(\xi_1, \eta_1) = a, \ H_2(\xi_2, \eta_2) = 2 \}.
\end{equation}
We consider its time-parametrization
\begin{equation}\label{Gamma0}
\Gamma_0^{\pm} := \{ (\xi_1^*(\tau_1), \eta_1^*(\tau_1), q_h(\tau_2), p_h^{\pm}(\tau_2)) : 
\tau_1, \tau_2\in \R \},
\end{equation}
where
\[
(q_h(s), p_h^{\pm}(s)) 
= \Big( 2 \arcsin (\tanh (s)), \, \pm \frac{2}{\cosh(s)} \Big).
\]
Now we prove that the periodic orbit $\mathcal{P}$ persists 
when $0 < \sigma \ll 1$. More precisely, we prove the following result.

\begin{proposition}\label{prop:hyppo}
Let $\mathcal{P}$ be the $T$-periodic orbit of \eqref{unpert.syst} defined in \eqref{def.mP.t.a}.
Then, there exist constants ${\sigma_1}>0$, $C > 0$ 
such that, for all $0<\sigma<\s_1$, there exists a $T$-periodic solution $\mathcal{P}_{\sigma}(t)$ 
of \eqref{syst.xi.eta.fully.normalized2} 
which is $\sigma$-close to $\mathcal{P}$ in the $C^1$-topology, namely
\[
\| \mP_\sigma - \mP \|_{C^1(\R)} \leq C \sigma.
\] 
Moreover $\mathcal{P}_{\sigma}$ possesses one stable and one unstable hyperbolic direction.
 \end{proposition}

Recall that  system \eqref{syst.xi.eta.fully.normalized2}  has the energy \eqref{mE.sigma} as   first integral. Then, the existence of a hyperbolic periodic orbit at each energy level is a consequence of classical perturbation theory. However, Proposition \ref{prop:hyppo} gives the existence of a periodic orbit for a fixed period. This could be shown by proving that the period is monotone with $\mathcal{E}$. Below, to make this paper selfcontained, we give an alternative proof of Proposition \ref{prop:hyppo} based on  a symmetry argument.

\begin{proof}[Proof of Proposition \ref{prop:hyppo}]
To prove the persistence of the periodic orbit $\mathcal{P}$ we use the fact that the system \eqref{syst.xi.eta.fully.normalized2} is reversible with respect to the involution
\[
\rho\colon (\T\times \R)^2\to (\T\times \R)^2 \qquad \rho(\xi_1, \eta_1, \xi_2, \eta_2)=(-\xi_1, \eta_1, -\xi_2, \eta_2).
\]
This means that, if we denote by $X$ the vector field of \eqref{syst.xi.eta.fully.normalized2}, then
$X \circ \rho=-\rho_* X$, 
where $\rho_*$ is the differential of $\rho$, which acts on the tangent space $\R^4$.
To apply an implicit function theorem argument it is convenient to pass to action-angle coordinates on the first pendulum (plane $(\xi_1, \eta_1)$). This will simplify the analysis of the linearized problem in the tangential directions at the periodic orbit $\mathcal{P}$. 
In the domain  
\[
S:=\left\{ (x_1, y_1)\in \T\times \R : H_1(x_1, y_1)\in (0, 2) \right\},
\]
we consider the action-angle variables  transformation
\[
\Phi\colon \T\times \mathcal{I}\to  S, \qquad (\xi_1, \eta_1)=\Phi(\theta, I)=(f(\theta, I), g(\theta, I))
\]
for some open interval $\mathcal{I}\subset \R$.
If we fix $(\xi_1, \eta_1)$ and call $2 \kappa=H_1(\xi_1, \eta_1)$, we can express $f$ and $g$ using elliptic functions in the following way:
\begin{equation}\label{def:aa}
\begin{cases}
f(\theta, I)=2 \arcsin\left( \sqrt{\kappa}\, \mathrm{sn}\left( \dfrac{2\mathtt{K}(\kappa)}{\pi} \theta\, \Big{|}\, \kappa \right) \right)\\[4mm]
g(\theta, I)=2 \sqrt{\kappa}\,\,\mathrm{cn} \left( \dfrac{2\mathtt{K}(\kappa)}{\pi} \theta\, \Big{|}\, \kappa \right),
\end{cases}
\end{equation}
where $\mathtt{K}(\kappa)$ is the complete elliptic integral of first kind and $\mathrm{sn}$ and $\mathrm{cn}$ are the elliptic sine and the elliptic cosine respectively (see e.g.\ \cite{Greenhill}). 



We now drop the sub-index from $\xi_2, \eta_2$.
Let us denote by 
\[
\Psi=(\Phi, \mathrm{Id})\colon \T\times \mathcal{I}\times \T\times \R\to S \times  \T\times \R, \quad\  \Psi(\theta, I, \xi, \eta)=(\Phi(\theta, I), \xi, \eta).
\]
The involution $\rho$ expressed in these new coordinates is just given by $\nu=\Psi^{-1}\circ\rho\circ\Psi$. Since the elliptic sine is odd and the elliptic cosine is even (with respect to its first variable), it is straightforward to see that system \eqref{syst.xi.eta.fully.normalized2} in the new coordinates is reversible with respect to the involution
\[
\nu(\theta, I, \xi, \eta)=(-\theta, I, -\xi, \eta).
\]
\medskip
We denote by $I\mapsto \Omega(I)$ the action-to-frequency map of the unperturbed pendulum.
\begin{remark}\label{rem:freqampl}
Note that $\Omega'(I)\neq 0$ for all $I\in \mathcal{I}$.
\end{remark}

The unperturbed periodic orbit $\mathcal{P}$ now reads as 
\[
\theta(t)=\Omega(I_0) t, \quad I(t)=I_0, \quad \xi(t)=\pi, \quad \eta(t)=0
\]
for some $I_0\in \mathcal{I}$. We remark that $\Omega(I_0)=\om$, where $\om=2\pi/T$ was defined below \eqref{def:T0}.
We consider the scaled time $t \rightsquigarrow\om t$, 
and the system \eqref{syst.xi.eta.fully.normalized2} becomes
\begin{equation}\label{syst.xi.eta.fully.normalized*} 
\begin{cases}
\om \dot{\theta} & = \Omega(I) +\s \mathcal{R}_1(\theta, I, \eta), \\
\om \dot{I} & = \s \mathcal{R}_2(\theta, I, \eta),  \\
\om\dot \xi & = \eta +\s \mathcal{R}_3(\theta, I),  \\
\om\dot \eta & = - \sin(\xi)
\end{cases}
\end{equation}
where 
 the functions $\mathcal{R}_i$, $i=1, 2, 3$, are determined by the relation
\[
[D \Psi (\theta, I, \xi, \eta)]^{-1} \begin{pmatrix}
- (\frac13 + \tilde \mu_1(\s)) \eta \\
0 \\
-(1 + \tilde \mu_2(\s)) g(\theta, I) \\
0
\end{pmatrix} =\begin{pmatrix}
\mathcal{R}_1(\theta, I, \eta)\\
\mathcal{R}_2(\theta, I, \eta)\\
\mathcal{R}_3(\theta, I)\\
0
\end{pmatrix}.
\]
We look for $2\pi$-periodic, smooth, reversible solutions of \eqref{syst.xi.eta.fully.normalized*}, namely $u(t)=(\theta(t), I(t), \xi(t), \eta(t))$ such that $\nu u(-t)=u(t)$. In other words, we look for solutions in the invariant set
\[
C^1_{\mathrm{odd, \,even}}(\T):=\{(\theta(t), I(t), \xi(t), \eta(t))\in C^1(\T; \T\times\mathcal I \times\T\times\R) : \theta(t), \xi(t)\,\,\mathrm{odd},\,\,I(t), \eta(t)\,\,\mathrm{even}  \}.
\]
Similarly, we define the space
\[
C^0_{\mathrm{even, \,odd}}(\T) :=\{(g_1(t), g_2(t), g_3(t), g_4(t))\in C^0(\T; \R^4) : g_1(t), g_3(t)\,\,\mathrm{even},\,\,g_2(t), g_4(t)\,\,\mathrm{odd}  \}.
\]

\begin{remark}
We note that $2\pi$-solutions of \eqref{syst.xi.eta.fully.normalized*} correspond to $T$-periodic solutions of \eqref{syst.xi.eta.fully.normalized2}.
\end{remark}

Let us define
\[
\mathcal{F}(\s; \cdot )\colon C^1_{\mathrm{odd, \,even}}(\T)\to C^0_{\mathrm{even, \,odd}}(\T),  \qquad 
\mathcal{F}(\s; \theta, I, \xi, \eta)=\begin{pmatrix}
\om\dot{\theta} - \Omega(I) -\s \mathcal{R}_1(\theta, I, \eta)
\notag \\
\om \dot{I} - \s \mathcal{R}_2(\theta, I, \eta)
\notag \\
\om\dot \xi - \eta -\s \mathcal{R}_3(\theta, I)
\notag \\
\om\dot \eta + \sin(\xi)
\end{pmatrix}.
\]
Since the maps $\Omega$, $\Psi$ and the vector field $X$ are analytic, we have that $\mathcal{F}(\s; \cdot )$ is at least $C^1$.
Then
\[
\mathcal{F}(0; \om t, I_0, \pi, 0)=0.
\]
We now study the linearized system at the unperturbed solution. We fix $(g_1, g_2, g_3, g_4)\in C^0_{\mathrm{even, \,odd}}(\T)$ and we look for solutions of the linear system

\begin{equation}\label{syst:lin}
\begin{cases}
\om \partial_t{\theta}-\Omega'(I_0) I=g_1\\
\om \partial_t{I}=g_2 \\
\om \partial_t{\xi}-\eta=g_3\\
\om \partial_t{\eta}-\xi=g_4.
\end{cases}
\end{equation}
We observe that the above system is decoupled, hence we can study separately the equations for $(\theta, I)$ and the ones for $(\xi, \eta)$.
Concerning the former, we first solve the equation for the actions. Since $g_2$ is odd we have $\langle g_2 \rangle = 0$ and
\[
I-\langle I \rangle=(\omega \partial_t)^{-1} g_2,
\]
where we denote by $\langle \cdot \rangle$ the time average over $\T$ and we denote by $(\omega \partial_t)^{-1} g_2$ the primitive of $g_2$ with zero average. Hence $I$ is determined up to its average. Substituting in the equation for the angle we obtain
\begin{equation}\label{eq:angles}
\om \partial_t \theta=\Omega'(I_0) (\omega \partial_t)^{-1} g_2 +\Omega'(I_0) \langle I \rangle+g_1.
\end{equation}
Equation \eqref{eq:angles} can be solved only if the r.h.s.\ has zero average. Therefore we fix
\[
\langle I \rangle=-\frac{\langle g_1 \rangle}{\Omega'(I_0)}, \quad \ 
\theta = (\omega \partial_t)^{-1} \big[ \Omega'(I_0) (\omega \partial_t)^{-1} g_2 +\Omega'(I_0) \langle I \rangle+g_1 \big]
\]
(note that $\Omega'(I_0)\neq0$, see Remark \ref{rem:freqampl}).
Concerning the equations for the $(\xi, \eta)$ variables we have the following:
by setting $v:=\xi+\eta$, $w:=\xi-\eta$, $h:=g_3+g_4$ and $\tilde{h}:=g_3-g_4$ we have that 
\begin{equation}\label{syst:vw}
\om\partial_t{v}-v=h, \qquad
\om \partial_t{w}+w=\tilde{h}.
\end{equation}
By Fourier series, one has that
\[
v(t)=\sum_{k\in \Z}\frac{h_k}{\mathrm{i}\om k-1} e^{\mathrm{i} k t}, \qquad w(t)=\sum_{k\in \Z} \frac{\tilde{h}_k}{\mathrm{i}\om k+1} e^{\mathrm{i} k t}
\]
are the unique solutions of \eqref{syst:vw}. Since $\xi=(v+w)/2$ and $\eta=(v-w)/2$, we recover the $\xi, \eta$ components of the solution of \eqref{syst:lin}. From the explicit expression of the solutions we have that for $g_i\equiv 0$, $i=1, \dots, 4$, the only solution of \eqref{syst:lin} is zero. Moreover the solutions are of class $C^1$. This implies that $d \mathcal{F}(0; \om t, I_0, \pi, 0)$ is invertible and by the implicit function theorem there exists $\sigma_1>0$ and a $C^1$ function $\mathtt{g}\colon (-\sigma_1, \sigma_1)\to C^1_{\mathrm{odd, \,even}}(\T)$ such that
\begin{equation*}
\mathcal{F}(\s; \mathtt{g}(\sigma))=0 \quad
\forall \sigma\in (-\sigma_1, \sigma_1), \qquad 
\mathtt{g}(0; t)=(\om t, I_0, \pi, 0).
\end{equation*}

We call $\mathcal{P}_{\sigma}$ the periodic orbit $\mathtt{g}(\s)$ written in the original coordinates $(\xi_1,\eta_1,\xi_2,\eta_2)$.
We also notice that by \eqref{syst:vw} this orbit is hyperbolic in the $(\xi_2, \eta_2)$ directions. This concludes the proof.
\end{proof}


By classical theory of persistence of invariant manifolds, $\mathcal{P}_{\sigma}$ has stable and unstable invariant manifolds $W_{\sigma}^{s, u}(\mathcal{P}_{\sigma})$ that depend differentiably on $\sigma$. Moreover these manifolds can be locally parametrized as $C^1$ graphs over the unperturbed invariant manifold \eqref{unpinvman}.

\subsection{Transverse intersection of invariant manifolds}
\label{subsec:Melnikov.int}


In this section we prove that the stable and unstable invariant manifolds $W_{\sigma}^{s, u}(\mathcal{P}_{\sigma})$ intersect transversally at some point. 
Since the invariant manifolds have dimension $2$ and we look for intersections within a $3$-dimensional energy level (see \eqref{mE.sigma}), it is sufficient to construct a $1$-dimensional section $\Lambda$ and measure the distance between the manifolds on the projection of $\Lambda$.

We recall the time parameterization of the unperturbed separatrix $\Gamma^{\pm}_0$ given in \eqref{Gamma0}. By symmetry we can consider just a single branch of the unperturbed homoclinic manifold, say $\Gamma_0=\Gamma_0^{+}$.
Let us consider a point $z_0=z_0(\tau_1, \tau_2)\in \Gamma_0\subset\{ H_1= a, \, H_2 = 2 \}$ 
 and define the section
\[
\Lambda:=\{ z_0+\lambda\nabla H_1(z) : \lambda\in \R \},
\]
where we use the notation $H_1(\xi_1,\eta_1,\xi_2,\eta_2):=H_1(\xi_1,\eta_1)$.
The line $\Lambda$ passes through $z_0$ and it is normal to $\Gamma_0$.

By the continuous dependence of the invariant manifolds on the parameters, 
for $\sigma$ small enough, $\Lambda$ intersects transversally 
also $W^{s, u}_{\sigma}(\mathcal{P}_{\sigma})$ 
at two points $z_{\sigma}^{s, u}=z_{\sigma}^{s, u}(\tau_1, \tau_2)$. 
We use the unperturbed energy of the first pendulum $H_1$ to measure the distance between $z^s_{\sigma}$ and $z^u_{\sigma}$. Note that the gradient of $H_1$ never vanishes on $\Gamma_0$, hence it is a good measure of a displacement in the normal directions of $\Gamma_0$.
We define the distance
\begin{equation}\label{distanceH1}
H_1(z_{\sigma}^s)-H_1(z_{\sigma}^u)=\sigma\,M+O(\sigma^2),
\end{equation}
where the first order of this distance is given by
\[
M=M(\tau_1, \tau_2):=\frac{d}{d \sigma}_{|_{\sigma=0}} (H_1(z_{\sigma}^s)-H_1(z_{\sigma}^u)).
\]
By classical arguments (see for instance \cite{Robinson}) we have that the first order is given by the Melnikov integral
\[
M(\tau_1, \tau_2)=\int_{-\infty}^{+\infty} D H_1 (\Phi_H^t(\Gamma_0(\tau_1, \tau_2)))  [Y (\Phi_H^t(\Gamma_0(\tau_1, \tau_2)))]\,dt,
\]
where $Y$ is the first order in $\sigma$ of the perturbation of the system \eqref{syst.xi.eta.fully.normalized2}, namely
\[
Y(\xi_1, \eta_1, \xi_2, \eta_2):=(-\eta_2/3, 0, -\eta_1, 0),
\]
and $\Phi_H^t$ is the Hamiltonian flow of $H$ in \eqref{def.H1.H2}.
We observe that (recall \eqref{Gamma0})
\[
\Phi_H^t(\Gamma_0(\tau_1, \tau_2))=(\xi_1^*(\tau_1+t), \eta_1^*(\tau_1+t), q_h(\tau_2+t), p_h^+(\tau_2+t)).
\]
By the autonomous nature of system \eqref{syst.xi.eta.fully.normalized2}, the Melnikov integral depends just on one parameter.
We define $\tau:=\tau_2-\tau_1$ and we consider the reduced Melnikov integral
\[
\mathcal{M}(\tau):=-\frac{1}{3}\int_{-\infty}^{+\infty} p_h^+(\tau+s)\,\sin(\xi_1^*(s)) \,ds.
\]
Note that, for all $\tau_1\in\R$, one has $\mathcal{M}(\tau)= M(\tau_1,\tau_1+\tau)= M(0,\tau)$.
We now prove the following.

\begin{lemma}\label{lemma:a0}
There exists a universal constant $a_0 \in (0,2)$ 
such that, for all $a \in [a_0,2)$, 
the reduced Melnikov integral $\mathcal{M}(\tau)$ has a non-degenerate zero at $\tau=0$.
\end{lemma}

\begin{proof}
We observe that $\mathcal{M}(0)=0$ because $\xi_1^*(t)$ is an odd function, 
while $p_h^+(t)$ is even. 
The derivative of the reduced Melnikov integral at $\tau=0$ is
\[
\frac{d}{d\tau} \mathcal{M}(0) 
= \frac{1}{3} \int_{-\infty}^{+\infty} \sin(q_h(s)) \sin(\xi_1^*(s)) \,ds
= \frac{2}{3} \int_{0}^{+\infty} \sin(q_h(s)) \sin(\xi_1^*(s)) \,ds.
\]
We recall that $(\xi_1^*, \eta_1^*)$ in \eqref{def.mP.t.a} 
depends on the parameter $a$, 
and we explicitly indicate 
the dependence on $a$ of the integral we want to study, denoting
\[
J(a) := \int_0^{+\infty} \sin(q_h(s))\,\sin(\xi_1^*(s;a)) \,ds.
\]
We have to prove that $J(a)$ is non zero for some value of $a \in (0,2)$. 
By the classical theorem of continuous dependence on initial data for ODEs, 
one has the following pointwise convergence:
\[
\text{for every } s \in [0,\infty), 
\quad \lim_{a \to 2} \xi_1^*(s;a) = q_h(s) 
\]
(even more, the convergence is uniform on compact intervals). 
Hence, for every $s \in [0, \infty)$, 
$f_a(s)$ converges to $f_2(s)$ 
as $a \to 2$, where 
\[
f_a(s) := \sin(q_h(s)) \sin(\xi_1^*(s;a)), 
\quad \ 
f_2(s) := \sin^2(q_h(s)).
\]
Moreover, since $q_h(s) \in [0, \pi)$ for all $s \in [0,\infty)$,  
one has 
\[
|f_a(s)| 
\leq |\sin(q_h(s))| 
= \sin(q_h(s)) 
= \frac{2 \sinh(s)}{\cosh^2(s)} 
=: g(s) 
\quad \forall s \in [0, \infty),
\]
and $g \in L^1(0,\infty)$. 
Hence, by the dominated convergence theorem, 
\[
J_* := \lim_{a \to 2} J(a) 
= \int_0^\infty f_2(s) \, ds 
= \int_0^\infty g^2(s) \, ds.
\]
The limit $J_*$ is finite by the exponential decay of $g^2$, 
and it is positive because $g^2$ is positive. 
Hence there exists $a_0 \in (0,2)$ such that $| J(a) - J_* | \leq J_*/2$ 
for all $a \in [a_0,2)$, and the lemma is proved.  
\end{proof}

By the above lemma and the Implicit Function Theorem, 
for $\sigma>0$ small enough, there exists at least one zero 
of the distance \eqref{distanceH1}, with transverse intersection.

%

\subsection{Symbolic dynamics}

We introduce the section
\[
\Pi:=\{ (\xi_1, \eta_1, \xi_2, \eta_2) :  \xi_2=0, \ 
\eta_2 > 0, \ \mE(\xi_1, \eta_1, \xi_2, \eta_2) = \mE(\mP_\s) \},
\]
where $\mE(\mP_\s)$ is the value of the prime integral $\mE$ 
in \eqref{mE.sigma} at the solution $\mP_\s$.
This section is transverse to the unperturbed flow $(\sigma=0)$ 
at a point of $W_0(\mathcal{P})$ and so, for $\sigma>0$ small enough, also to the perturbed one. 
Moreover, by Lemma \ref{lemma:a0}, 
it contains points of $W^s_{\s} (\mathcal{P}_{\s}) \pitchfork W^u_{\s} (\mathcal{P}_{\s})$
(where $\pitchfork$ means transverse intersection).

Denote by $\Phi^t$ the flow of system \eqref{syst.xi.eta.fully.normalized2}. 
Fixed a point $z\in \Pi$, we define $\mathcal{T}(z)>0$ as the first 
(forward) return time to $\Pi$. 
For those points $z$ that do not hit back the section $\Pi$
(for instance, the points of $W_{\s}^s(\mathcal{P}_{\s})$), 
we set $\mathcal{T}(z)=+\infty$. We define the open set $\mathcal{U}\subset \Pi$ as
\[
\mathcal{U}:=\{ z\in \Pi : \mathcal{T}(z)<+\infty \}
\]
and the associated Poincar\'e map 
$\mathtt{P}\colon \mathcal{U}\subset\Pi\to \Pi$ by 
$\mathtt{P}(z)=\Phi^t(z)|_{t=\mathcal{T}(z)}$.

\begin{proposition}{(Smale Horseshoe)}
\label{prop:Smale}
There exist universal constants $\sigma_0, C_1, C_2 >0$ 
such that for all $\sigma\in (0, \sigma_0)$ 
there exists a positive integer $M_0(\s)$ in the interval 
\begin{equation} \label{log.M.0}
C_1 \log(\s^{-1}) 
\leq M_0(\s) 
\leq C_2 \log(\s^{-1})
\end{equation}
such that the Poincar\'e map $\mathtt{P}$ possesses an invariant set $Y\subset \mathcal{U}$ whose dynamics is conjugated to the infinite symbols shift. Namely, there exists a homeomorphism $h\colon \mathcal{A}\to Y$, 
where
\[
\mathcal{A} := 
\{ \omega=\{  \omega_k\}_{k\in \Z} : \omega_k\in \mathbb{N}, 
\ \ \om_k \geq M_0(\s) 
\ \forall k \in \Z   \},
\]
such that $\mathtt{P}_{|_Y}=h\circ \mathfrak{d} \circ h^{-1}$ where $\mathfrak{d} \colon \mathcal{A}\to \mathcal{A}$ is the shift
\[
(\mathfrak{d} \omega)_k=\omega_{k+1}, \quad \ k\in \Z.
\]
Moreover $h^{-1}$ can be defined as follows. 
Associated to $z \in Y$ one can define the sequence of hitting times
\begin{align*}
t_0=0, \quad \  t_k= t_{k-1} + \mathcal{T}(\mathtt{P}^{k-1} z) \quad \mathrm{for}\,\,\,k\geq 1,
\quad \  t_k= t_{k+1} - \mathcal{T}(\mathtt{P}^{k} z) \quad \mathrm{for}\,\,\,k\le -1,
\end{align*}
and $h^{-1}(z) := \omega = (\om_k)_{k \in \Z}$, with
\[
\omega_k = \left\lfloor \frac{t_k-t_{k-1}}{T} \right\rfloor,
\]
where $T = T_a$ is the period of the periodic orbit $\mathcal{P}_{\s}$, 
and $\lfloor \cdot \rfloor$ denotes the integer part.
\end{proposition}

\begin{proof}
The proof 
follows the same lines as the construction of symbolic dynamics 
done by Moser in Chapter 3 of \cite{Moserbook}.
\end{proof}

Proposition \ref{prop:Smale} concludes the proof of Proposition \ref{prop:chaos.xi.eta}.
Note that the return times to the section are large since orbits get close to the hyperbolic periodic orbit.

\section{Back to the truncated effective system} 
\label{sec:back}

In Section \ref{sec:chaos.2.pendulums} we have proved the existence 
of chaotic solutions for system \eqref{syst.xi.eta.fully.normalized2}. 
These solutions are global in time.
In the next lemma we obtain the corresponding solutions of the 
truncated effective system \eqref{syst.6.eq}.

\begin{lemma} \label{lemma:back.from.xi.eta.to.6.eq}
Let $a=a_0$ in Proposition \ref{prop:chaos.xi.eta}, 
and let $\s_0$ be the corresponding universal constant 
given by Proposition \ref{prop:chaos.xi.eta}.
Let $m,p$ be two integers, $2 \leq m < p$, 
with ratio $\s := m/p$ in the interval $(0,\s_0)$. 
Define $\a_1, \ldots \a_4$ by \eqref{alpha.1234.Fib},
$\tilde \mu_1(\s), \tilde \mu_2(\s)$ by \eqref{Taylor.mu.12.gamma},
and $\g$ by the second identity in \eqref{def.gamma}. 
Assume the hypotheses of Proposition \ref{prop:chaos.xi.eta}, 
and consider the solution $(\xi_1(t), \eta_1(t), \xi_2(t), \eta_2(t))$ 
of system \eqref{syst.xi.eta.fully.normalized2} 
obtained in Proposition \ref{prop:chaos.xi.eta}. 

Consider any three real numbers $\mathtt{A}_1, q_1, q_2$, with $\mathtt{A}_1 > 0$,
and define the following constants: 
define $c_{123}$ by \eqref{c123.from.A1}, 
define $c_{234} = c_{123} / \gamma$, 
define $\mathtt{B}$ by the second identity in \eqref{fix.A1.B},
define $\mathtt{A}_2$ by \eqref{fix.A2},
define $b_1, b_2$ by \eqref{b12.from.q12}, 
define $E_1, E_2$ by \eqref{E12.from.b12}.
Define the functions 
\begin{equation}\label{eq:systemmathtt}
\begin{aligned}
S_1(t) & := E_1 - q_1 - \mathtt{A}_1 \eta_1(\mathtt{B}t), 
\\
S_2(t) & := E_2 - q_1 - q_2 - \mathtt{A}_1 \eta_1(\mathtt{B}t) - \mathtt{A}_2 \eta_2(\mathtt{B}t),
\\
S_3(t) & := q_1 - q_2 + \mathtt{A}_1 \eta_1(\mathtt{B}t) - \mathtt{A}_2 \eta_2(\mathtt{B} t),
\\
S_4(t) & := q_2 + \mathtt{A}_2 \eta_2(\mathtt{B} t),
\\
\ph_{123}(t) & := \xi_1(\mathtt{B} t),
\\
\ph_{234}(t) & := \xi_2(\mathtt{B}t).
\end{aligned}
\end{equation}
Then $(S_1(t), S_2(t), S_3(t), S_4(t), \ph_{123}(t), \ph_{234}(t), c_{123}, c_{234})$ 
satisfies \eqref{syst.6.eq} for all $t \in \R$. 
\end{lemma}

\begin{proof} 
We apply 
Lemma \ref{lemma:change.xy.tilde.xi.eta}-$(ii)$ with $\lm = 1$ 
to go back from 
system \eqref{syst.xi.eta.fully.normalized2}
to system \eqref{syst.xy.tilde}, 
then
Lemma \ref{lemma:change.xy.xy.tilde}-$(ii)$
to go back from system \eqref{syst.xy.tilde} 
to system \eqref{syst.xy},
then 
Lemma \ref{lemma:change.4.eq.xy}-$(ii)$
to go back from system \eqref{syst.xy}
to system \eqref{syst.4.eq},
and finally 
Lemma \ref{lemma:change.6.eq.4.eq}-$(ii)$
to go back from system \eqref{syst.4.eq} 
to system \eqref{syst.6.eq}.
\end{proof}

%

\subsection{Positivity of the superactions}

Now we come to the question whether the solutions of system \eqref{syst.6.eq} 
obtained in Lemma \ref{lemma:back.from.xi.eta.to.6.eq}
satisfy the inequalities \eqref{nec.cond.S}.
As a first step, we study the constant terms 
\begin{equation} \label{constant.terms.in.S}
E_1 - q_1, \qquad 
E_2 - q_1 - q_2, \qquad 
q_1 - q_2, \qquad 
q_2
\end{equation}
appearing in the definition of $S_1, \ldots, S_4$ 
in Lemma \ref{lemma:back.from.xi.eta.to.6.eq}. 
We compute the formula of $E_1, E_2$ as functions of $q_1, q_2$: 
from \eqref{E12.from.b12} and \eqref{b12.from.q12} we get
\begin{equation} \label{E12.from.q12}
E_1 = \frac{\a_1^2 + 2 \a_3^2}{\a_1^2} q_1 - \frac{2\a_3^2 + \a_4^2}{\a_1^2} q_2,
\qquad 
E_2 = - \frac{\a_3^2 - \a_2^2}{\a_2^2} q_1 + \frac{\a_2^2 + \a_3^2 + \a_4^2}{\a_2^2} q_2.
\end{equation}
Next, we observe in the following lemma that $q_1, q_2$ can be chosen such that 
the constant terms \eqref{constant.terms.in.S} are all positive.

\begin{lemma}  \label{lemma:fix.ratio.q1.q2}
Let $q_1, q_2$ be positive real numbers with ratio $q_1 / q_2$ 
in the interval $(1 + \frac{r}{2}, 1 + r)$, $r = \a_4^2 / \a_3^2$, 
and let $E_1, E_2$ be defined by \eqref{E12.from.q12}.
Then 
the constants \eqref{constant.terms.in.S} are all positive.
In particular, if
\begin{equation} \label{fix.ratio.q1.q2}
q_1 = \Big( 1 + \frac23 r \Big) q_2,
\end{equation}
then
\begin{equation} \label{optimize.constant.terms.E.q}
E_1 - q_1 = \frac13 \frac{\a_4^2}{\a_1^2} q_2, 
\qquad 
E_2 - q_1 - q_2 = \frac13 \frac{\a_4^2}{\a_2^2} q_2,
\qquad 
q_1 - q_2 = \frac23 \frac{\a_4^2}{\a_3^2} q_2.
\end{equation}
\end{lemma}

\begin{proof} 
By \eqref{E12.from.q12}, 
\begin{align*}
E_1 - q_1 
& = \Big( \frac{2 \a_3^2}{\a_1^2} \, \frac{q_1}{q_2} 
- \frac{2\a_3^2 + \a_4^2}{\a_1^2} \Big) q_2, 
\qquad 
E_2 - q_1 - q_2 
= \Big( -  \frac{\a_3^2}{\a_2^2} \, \frac{q_1}{q_2} 
+ \frac{\a_3^2 + \a_4^2}{\a_2^2} \Big) q_2,
\\
q_1 - q_2 
& = \Big( \frac{q_1}{q_2} - 1 \Big) q_2.
\end{align*}
We write the ratio $q_1/q_2$ as 
$1 + \th r$, 
where $\th \in \R$ is a free parameter and $r = \a_4^2 / \a_3^2$.
Then 
\[
E_1 - q_1 = \frac{\a_4^2 q_2}{\a_1^2} (2\th - 1), 
\qquad 
E_2 - q_1 - q_2 = \frac{\a_4^2 q_2}{\a_2^2} (1-\th),
\qquad 
q_1 - q_2 = \frac{\a_4^2 q_2}{\a_3^2} \th.
\]
The minimum
\[
\min \{ 2\th - 1, 1 - \th, \th, 1 \} 
= \min \{ 2\th - 1, 1 - \th \} 
\]
is positive for $\th \in (\frac12, 1)$, 
and it reaches its maximum value at $\th = 2/3$. 
\end{proof}

Note that, by \eqref{alpha.1234.Fib}, 
the ratio $r = \a_4^2 / \a_3^2$ tends to $4$
as $\s = m/p \to 0$.
By \eqref{def.gamma}, \eqref{Taylor.mu.12.gamma}, the constant $\mathtt{A}_2$ 
in Lemma \ref{lemma:back.from.xi.eta.to.6.eq} satisfies 
\begin{equation} \label{formula.A2}
\mathtt{A}_2 = \frac{\mathtt{A}_1}{\gamma} 
= \frac{1 + 3\s + 3 \s^2}{3 + 9\s + 7 \s^2} \mathtt{A}_1 \leq \mathtt{A}_1.
\end{equation}
The solutions $\eta_1, \eta_2$,  constructed in Proposition \ref{prop:chaos.xi.eta} and appearing in Lemma \ref{lemma:back.from.xi.eta.to.6.eq},  satisfy
\begin{equation} \label{eta.12.bound}
\sup_{t \in \R} |\eta_1(t)| \leq 3, \qquad 
\sup_{t \in \R} |\eta_2(t)| \leq 3
\end{equation}
(more accurate estimates about $\eta_1, \eta_2$ 
have been obtained in Proposition \ref{prop:chaos.xi.eta}).
Thus, we prove the following bound for $S_n$ from below.

\begin{lemma} \label{lemma:bound.S.below}
Let $q_2, \mathtt{A}_1$ be any two positive real numbers, 
and define $q_1$ by \eqref{fix.ratio.q1.q2}. 
If 
\begin{equation} \label{smallness.A1.q2}
\mathtt{A}_1 \leq \frac{q_2}{9}, 
\end{equation}
then the functions $S_n(t)$ defined in Lemma \ref{lemma:back.from.xi.eta.to.6.eq} 
satisfy 
\begin{equation} \label{bound.S.below}
S_n(t) \geq \frac{q_2}{2} > 0  
\qquad \forall t \in \R\quad \text{and}\quad \forall n=1,2,3,4.
\end{equation}
\end{lemma}

\begin{proof}
By \eqref{optimize.constant.terms.E.q} 
and  \eqref{formula.A2},
the functions $S_n(t)$ defined in Lemma \ref{lemma:back.from.xi.eta.to.6.eq} 
satisfy for all $t \in \R$
\begin{align*}
S_1(t) 
& \geq E_1 - q_1 - 3 \mathtt{A}_1 
= \frac13 \frac{\a_4^2}{\a_1^2} q_2 - 3 \mathtt{A}_1, 
\\
S_2(t) 
& \geq E_2 - q_1 - q_2 - 3 \mathtt{A}_1 - 3 \mathtt{A}_2
\geq \frac13 \frac{\a_4^2}{\a_2^2} q_2 - 6 \mathtt{A}_1,
\\
S_3(t) 
& \geq q_1 - q_2 - 3 \mathtt{A}_1 - 3 \mathtt{A}_2
\geq \frac23 \frac{\a_4^2}{\a_3^2} q_2 - 6 \mathtt{A}_1, 
\\
S_4(t) 
& \geq q_2 - 3 \mathtt{A}_2 
\geq q_2 - 3 \mathtt{A}_1.
\end{align*} 
By \eqref{alpha.1234.Fib}, one has  
\[
\frac{\a_4^2}{\a_1^2} > 4, \qquad 
\frac{\a_4^2}{\a_2^2} > 4, \qquad 
\frac{2 \a_4^2}{\a_3^2} > 4,
\]
which, together with \eqref{smallness.A1.q2}, implies \eqref{bound.S.below}.
\end{proof}

By Lemma \ref{lemma:bound.S.below},  
the  condition \eqref{smallness.A1.q2} 
implies the meaningfulness condition \eqref{nec.cond.S} 
for the solutions of system \eqref{syst.6.eq}. 
By Lemma \ref{lemma:bound.S.below}, 
now we have two free parameters, which are $q_2$ and $\mathtt{A}_1$, 
related by the inequality $0 < \mathtt{A}_1 \leq q_2/9$, 
while $q_1$ is now constrained by formula \eqref{fix.ratio.q1.q2}. 

\subsection{Initial data in the normal form ball} 

The special solutions of system \eqref{syst.6.eq} 
defined in Lemma \ref{lemma:back.from.xi.eta.to.6.eq}
will be compared, by a Gronwall argument, 
with those of the full (i.e., non-truncated) effective system \eqref{0906.11},
\eqref{3105.11} starting at the same initial data 
at time $t=0$. 
The initial data we are interested in correspond to functions 
$u_0(x)$ in the ball \eqref{ball.mitica},
because every $u_0$ in that ball produces a solution 
of the Cauchy problem \eqref{3101.16}, \eqref{initial.cond.u.v.sec.prepar}
that remains, for a sufficiently long interval of time, 
in the domain where the normal form transformation is well-defined,
as is explained quantitatively in Lemma \ref{lemma:total}. 
Recall that 
$\d_1 > 0$ in \eqref{ball.mitica} is a universal constant, 
and $m_1$ is defined in \eqref{def.m1}.
The following lemma deals with the ball \eqref{ball.mitica} written in terms of $S_n$.


\begin{lemma} \label{lemma:ball.m1.S}
Consider the solutions of system \eqref{syst.6.eq}
given by Lemma \ref{lemma:back.from.xi.eta.to.6.eq}, 
and assume \eqref{fix.ratio.q1.q2}, 
\eqref{smallness.A1.q2}. 
Then 
\begin{equation}  \label{bound.norm.s.S}
\sum_{n=1}^4 \a_n^{2 s} S_n(t) 
\leq 8 \a_4^{2s} q_2
\quad \forall t \in \R, 
\quad s \in [1,\infty).
\end{equation}
\end{lemma}

\begin{proof}
For all $t \in \R$, 
if the smallness condition \eqref{smallness.A1.q2} is satisfied, 
then the solutions of system \eqref{syst.6.eq} 
obtained in Lemma \ref{lemma:back.from.xi.eta.to.6.eq}
satisfy 
\begin{align*}
\sum_{n=1}^4 \a_n^{2 s} S_n
& \leq \a_1^{2s} (E_1 - q_1 + 3 \mathtt{A}_1) 
+ \a_2^{2s} (E_2 - q_1 - q_2 + 6 \mathtt{A}_1)
+ \a_3^{2s} (q_1 - q_2 + 6 \mathtt{A}_1)
+ \a_4^{2s} (q_2 + 3 \mathtt{A}_1) 
\\
& \leq \a_1^{2s} \frac{2 \a_4^2 q_2}{\a_1^2} 
+ \a_2^{2s} \frac{2 \a_4^2 q_2}{\a_2^2} 
+ \a_3^{2s} \frac{2 \a_4^2 q_2}{\a_3^2} 
+ \a_4^{2s} 2 q_2  \leq 8 \a_4^{2s} q_2,
\end{align*}
where we have used the bounds in \eqref{eta.12.bound} for $\eta_1, \eta_2$, 
the identities \eqref{optimize.constant.terms.E.q} 
for the constants terms \eqref{constant.terms.in.S}, 
the bound \eqref{smallness.A1.q2} for $\mathtt{A}_1$, 
the bound \eqref{formula.A2} for $\mathtt{A}_2$, 
and the fact that $s\geq 1$, $\alpha_1<\dots<\alpha_4$.
\end{proof}

A consequence of this lemma is the following. If one chooses $q_2$ such that
\begin{equation} \label{smallness.q2}
	8 \a_4^{2m_1} q_2 \leq \d_1^2,
\end{equation}
where $m_1, \d_1$ are the universal constants of the ball \eqref{ball.mitica}, then 
\begin{equation}  \label{ball.m1.S}
\sum_{n=1}^4 \a_n^{2 m_1} S_n(t) 
\leq \d^2  \qquad \forall t \in \R.
\end{equation}

\subsection{Construction of a compatible initial datum}

Given a trigonometric polynomial $u \in C(\T^d,\C)$, 
Fourier supported on the set $\{ k \in \Z^d : |k| \in \{ \a_1, \a_2, \a_3, \a_4 \} \}$,
with Fourier coefficients $u_k$, 
we use the superscript $u$ to denote 
\begin{equation} \label{def.SB.(u)}
S_n^{(u)} := \sum_{|k| = \a_n} |u_k|^2, 
\qquad 
B_n^{(u)} := \sum_{|k| = \a_n} u_k u_{-k},
\qquad 
n = 1,2,3,4,
\end{equation}
\begin{equation} \label{def.Z.rho.c.(u)}
Z_{123}^{(u)} := B_1^{(u)} B_2^{(u)} \overline{ B_3^{(u)} },
\qquad
\rho_{123}^{(u)} := |Z_{123}^{(u)}|, 
\qquad 
c_{123}^{(u)} := \frac38 \rho_{123}^{(u)} \a_1 \a_2 \a_3,
\end{equation}
analogous definitions for $Z_{234}^{(u)}$, 
$\rho_{234}^{(u)}$, $c_{234}^{(u)}$, 
and, if $\rho_{123}^{(u)}, \rho_{234}^{(u)}$ are positive, 
we define the angles 
$\ph_{123}^{(u)}$, $\ph_{234}^{(u)} \in \T = \R / 2 \pi \Z$ 
by the identities
\begin{equation} \label{def.polar.(u)}
Z_{123}^{(u)} = \rho_{123}^{(u)} \exp(i \ph_{123}^{(u)} ),
\qquad 
Z_{234}^{(u)}= \rho_{234}^{(u)} \exp(i \ph_{234}^{(u)} ).
\end{equation}
We consider the following question: 

\emph{Given a solution 
$(S_1(t), S_2(t), S_3(t), S_4(t), \ph_{123}(t), \ph_{234}(t))$ 
of system \eqref{syst.6.eq} 
obtained in Lemma \ref{lemma:back.from.xi.eta.to.6.eq}, 
and taken, in particular, its value at time $t=0$, 
does there exist a function $u_0(x)$ in the ball \eqref{ball.mitica},
Fourier supported on the spheres of radius $\a_1, \a_2, \a_3, \a_4$, 
such that 
\begin{equation} \label{good.u0.question}
\begin{cases} 
S_n^{(u_0)} = S_n(0), \quad n = 1,2,3,4, 
\\
\ph_{123}^{(u_0)} = \ph_{123}(0), 
\quad 
\ph_{234}^{(u_0)} = \ph_{234}(0), 
\\
c_{123}^{(u_0)} = c_{123}, 
\quad 
c_{234}^{(u_0)} = c_{234} \quad ?
\end{cases}
\end{equation}}
The equations for the angles must be interpreted 
as identities of elements of $\T = \R / 2\pi\Z$. 
The affirmative answer to this question is given 
in Lemma \ref{lemma:good.u0} below, 
whose proof uses the next three simple preparatory lemmas.

\begin{lemma} \label{lemma:exist.numbers}
Let 
$(S_1(t), S_2(t), S_3(t), S_4(t), \ph_{123}(t), \ph_{234}(t))$ 
be a solution of system \eqref{syst.6.eq} 
obtained in Lemma \ref{lemma:back.from.xi.eta.to.6.eq}. 
Assume that $q_1, q_2, \mathtt{A}_1$ satisfy 
\eqref{fix.ratio.q1.q2}, \eqref{smallness.A1.q2}. 
If, in addition, $\mathtt{A}_1, q_2$ satisfy
\begin{equation} \label{smallness.A1.q2.power.32}
\mathtt{A}_1 \leq \Big( \frac{3}{32} \, \frac{\a_1^3}{\a_1^2 + \a_2^2 + \a_3^2} \Big)^{\frac12} 
q_2^{\frac32}, 
\end{equation}
then there exist real numbers $r_n, \psi_n$, $n = 1,2,3,4$, with 
\begin{equation} \label{ineq.r.S}
0 < r_n \leq S_n(0), 
\end{equation}
such that 
\begin{alignat}{2}
\psi_1 + \psi_2 - \psi_3 & = \ph_{123}(0), 
\quad & \quad 
\frac38 (r_1 \a_1)(r_2 \a_2)(r_3   \a_3) & = c_{123}, 
\label{psi.ph.r.c.123}
\\
\psi_2 + \psi_3 - \psi_4 & = \ph_{234}(0), 
\quad & \quad
\frac38 (r_2 \a_2)(r_3\a_3)  (r_4   \a_4) & = c_{234}.
\label{psi.ph.r.c.234}
\end{alignat}
\end{lemma}

\begin{proof}
Regarding $(\psi_1, \psi_2, \psi_3, \psi_4)$, 
there are infinitely many solutions, 
because they are 4 unknowns that have to satisfy just 2 linear constraints.
For example, we can fix 
$(\psi_1, \psi_2, \psi_3, \psi_4) = (\ph_{123}(0), 0, 0, - \ph_{234}(0))$.
Regarding $r_n$, we first recall that, from Lemma \ref{lemma:back.from.xi.eta.to.6.eq}, 
the constants $c_{123}$, $c_{234}$ are 
\[
c_{123} = \frac{\a_1^2 + \a_2^2 + \a_3^2}{2} \tA_1^2, 
\qquad 
c_{234} = \frac{ c_{123} }{ \g },
\]
where $\gamma$ is defined in \eqref{def.gamma} 
and satisfies $1 < \gamma < 3$.
Therefore $c_{123}, c_{234}$ are positive, 
because $\tA_1$ is positive. We have to choose $r_n$ such that 
\begin{align*}
(\a_1 r_1) (\a_2 r_2) (\a_3 r_3) 
& = \frac83 c_{123} 
= \frac43 (\a_1^2 + \a_2^2 + \a_3^2) \tA_1^2, 
\\
(\a_2 r_2) (\a_3 r_3) (\a_4 r_4) 
& = \frac83 \frac{ c_{123} }{ \g } 
= \frac43 (\a_1^2 + \a_2^2 + \a_3^2) \tA_1^2 \frac{1}{\gamma}.
\end{align*}
Hence we fix 
\[
r_1 = \frac{r_0}{\a_1}, \quad
r_2 = \frac{r_0}{\a_2}, \quad
r_3 = \frac{r_0}{\a_3}, \quad
r_4 = \frac{r_0}{\gamma \a_4}, \qquad
r_0 := \Big( \frac43 (\a_1^2 + \a_2^2 + \a_3^2) \tA_1^2 \Big)^{\frac13}.
\]
Thus, $r_n$ are all positive and satisfy the required identities.
It only remains to check that $r_n \leq S_n(0)$. 
By \eqref{bound.S.below}, we know that $q_2 / 2 \leq S_n(0)$ for all $n=1,2,3,4$. 
Since $\a_1 < \a_2 < \a_3 < \a_4$ and $\g > 1$, 
we have $r_4 < r_3 < r_2 < r_1$. Hence it is sufficient to check that $r_1 \leq q_2/2$, 
and, by the definition of $r_1 = r_0 / \a_1$, 
this holds if $\tA_1, q_2$ satisfy \eqref{smallness.A1.q2.power.32}.
\end{proof}

\begin{lemma} \label{lemma:basic.complex}
Let $s,r,\psi$ be real numbers such that $0 < r \leq s$. 
Then there exists $z_1, z_2 \in \C \setminus \{ 0 \}$ 
such that 
$|z_1|^2 + |z_2|^2 = s$ and $2 z_1 z_2 = r e^{i\psi}$.
\end{lemma}



\begin{lemma} \label{lemma:u0.from.s.r.psi}
Let $s_n, r_n, \psi_n$, $n=1,2,3,4$, be real numbers such that 
$0 < r_n \leq s_n$. Then, in any dimension $d \geq 1$, 
there exists a trigonometric polynomial $u_0 \in C(\T^d,\C)$, 
Fourier supported on the set $\{ k \in \Z^d : |k| \in \{ \a_1, \a_2, \a_3, \a_4 \} \}$,
such that, recalling the notation \eqref{def.SB.(u)}, 
\begin{equation} \label{S.s.B.r.psi}
S_n^{(u_0)} = s_n, \quad 
B_n^{(u_0)} = r_n e^{i \psi_n}, \quad 
n=1,2,3,4.
\end{equation}
\end{lemma}

\begin{proof} 
For each $n = 1,2,3,4$, fix an integer vector $k_n \in \Z^d$ with $|k_n| = \a_n$, 
and apply Lemma \ref{lemma:basic.complex} to determine two nonzero complex numbers 
$z_{1,n}, z_{2,n}$ such that 
\[
|z_{1,n}|^2 + |z_{2,n}|^2 = s_n, \qquad 
2 z_{1,n} z_{2,n} = r_n e^{i \psi_n}.
\]
We define $u_0$ as the trigonometric polynomial having 
$z_{1,n}, z_{2,n}$ as Fourier coefficients for the frequencies $k_n, -k_n$, 
and having no other frequencies in its support, i.e.
\[
u_0(x) = \sum_{n=1}^4 ( z_{1,n} e^{i k_n \cdot x} + z_{2,n} e^{- i k_n \cdot x} ).
\]
Then 
\[
S_n^{(u_0)} = |z_{1,n}|^2 + |z_{2,n}|^2 = s_n, 
\qquad 
B_n^{(u_0)} = 2 z_{1,n} z_{2,n} = r_n e^{i \psi_n}.
\qedhere
\]
\end{proof}

Lemma \ref{lemma:u0.from.s.r.psi} deals with 
trigonometric polynomials supported on just one pair $(k_n, -k_n)$ 
of points on the sphere $\{ k : |k| = \a_n \}$; 
this is the minimal situation, valid in any dimension $d \geq 1$. 
Of course, in dimension $d \geq 2$ the Fourier support 
can contain more than one pair of opposite frequencies 
on the same sphere, and therefore the construction of $u_0$ 
with prescribed $S_n^{(u_0)}, B_n^{(u_0)}$ has even more free parameters
at disposal.

The following lemma gives the answer to question \eqref{good.u0.question}.

\begin{lemma} \label{lemma:good.u0}
Let 
$(S_1(t), S_2(t), S_3(t), S_4(t), \ph_{123}(t), \ph_{234}(t))$ 
be a solution of system \eqref{syst.6.eq} 
obtained in Lemma \ref{lemma:back.from.xi.eta.to.6.eq}.
Also assume that $q_1, q_2, \tA_1$ satisfy 
\eqref{fix.ratio.q1.q2}, 
\eqref{smallness.A1.q2},  
\eqref{smallness.A1.q2.power.32}.
Then there exists a trigonometric polynomial $u_0 \in C(\T^d,\C)$, 
Fourier supported on the set $\{ k \in \Z^d : |k| \in \{ \a_1, \a_2, \a_3, \a_4 \} \}$,
satisfying all the identities in \eqref{good.u0.question}. 
If, in addition, $q_2$ satisfies \eqref{smallness.q2},  
then $u_0$ belongs to the ball \eqref{ball.mitica}.
\end{lemma}

\begin{proof}
By Lemma \ref{lemma:exist.numbers}, 
there exist constants $r_n, \psi_n$ satisfying 
\eqref{ineq.r.S},
\eqref{psi.ph.r.c.123},
\eqref{psi.ph.r.c.234}. 
Define $s_n := S_n(0)$, 
so that \eqref{ineq.r.S} becomes $0 < r_n \leq s_n$.  
By Lemma \ref{lemma:u0.from.s.r.psi}, 
there exists a trigonometric polynomial $u_0$, 
with the desired Fourier support, 
satisfying \eqref{S.s.B.r.psi}. 
By the first identity in \eqref{S.s.B.r.psi} we directly have
\[
S_n^{(u_0)} = s_n = S_n(0), \qquad n = 1,2,3,4.
\]
By the second identity in \eqref{S.s.B.r.psi}, 
the first definition in \eqref{def.Z.rho.c.(u)}, 
and the first identity in \eqref{psi.ph.r.c.123},
we obtain 
\begin{equation} \label{temp.Z.r.ph}
Z_{123}^{(u_0)} = r_1 r_2 r_3 \exp( i (\psi_1 + \psi_2 - \psi_3) )
= r_1 r_2 r_3 \exp ( i \ph_{123}(0) ).
\end{equation}
By \eqref{temp.Z.r.ph}, 
by the second and third definition in \eqref{def.Z.rho.c.(u)}, 
and by the second identity in \eqref{psi.ph.r.c.123}, we get
\[
\rho_{123}^{(u_0)} = r_1 r_2 r_2,  
\qquad 
c_{123}^{(u_0)} = \frac38 r_1 r_2 r_3 \a_1 \a_2 \a_3
= c_{123}.
\]
Moreover, since $r_1 r_2 r_3 > 0$, \eqref{temp.Z.r.ph} 
is a polar representation of $Z_{123}^{(u_0)}$, 
and hence 
$\ph_{123}^{(u_0)} = \ph_{123}(0)$
as elements of $\T = \R / 2\pi \Z$. 
Similar proof applies for $c_{234}$, $\ph_{234}(0)$. 

Finally, if $q_2$ satisfies \eqref{smallness.q2}, 
then, by Lemma \ref{lemma:ball.m1.S}, bound \eqref{ball.m1.S} holds; 
in particular, this bound 
at time $t=0$ implies that $u_0$ belongs to the ball \eqref{ball.mitica}.
\end{proof}

\subsection{Joining the two amplitude parameters}

In Lemma \ref{lemma:back.from.xi.eta.to.6.eq} 
the solutions of system \eqref{syst.6.eq} 
obtained from Proposition \ref{prop:chaos.xi.eta}
are described by the three independent parameters $q_1, q_2,\tA_1$. 
Then, to get the positivity of the functions $S_n$, it is enough to use that
the ratio $q_1/q_2$ is bounded from below and from above by \eqref{fix.ratio.q1.q2}. 
After that, only two independent parameters remain, which are $q_2$ and $\tA_1$. 
Then, to obtain the lower bound \eqref{bound.S.below}, 
we need \eqref{smallness.A1.q2}, which is a bound of the form 
$\tA_1 \leq C q_2$ for some universal constant $C$. 
Also, $q_2$ itself must satisfy the smallness condition \eqref{smallness.q2}, 
which is an inequality of the form $q_2 \leq K$, 
for some constant $K$ depending on $m,p$. 
Next, $\tA_1, q_2$ must also satisfy the condition \eqref{smallness.A1.q2.power.32}, 
which is an inequality of the form $\tA_1 \leq K q_2^{3/2}$, 
for some constant $K$ depending on $m,p$. 

We would like to obtain values of $\tA_1$ as large as possible,
because $\tA_1$ and its multiple $\tA_2$ are the amplitudes of the chaotic movements 
we want to construct. 
We fix $\tA_1$ as the largest value compatible with 
\eqref{smallness.A1.q2} and \eqref{smallness.A1.q2.power.32}. 
Thus, we define 
\begin{equation} \label{def.eps}
\e := q_2^{\frac12}, \qquad 
\tA_1 := \Big( \frac{3}{32} \frac{\a_1^3}{\a_1^2 + \a_2^2 + \a_3^2} \Big)^{\frac12} \e^3,
\end{equation}
so that \eqref{smallness.A1.q2.power.32} is satisfied. 
Note that \eqref{smallness.A1.q2} becomes 
\begin{equation} \label{smallness.A1.q2.in.terms.of.eps}
\Big( \frac{3}{32} \frac{\a_1^3}{\a_1^2 + \a_2^2 + \a_3^2} \Big)^{\frac12} \e^3 
\leq \frac{1}{9} \e^2.
\end{equation}
We define 
\begin{equation} \label{def.eps.0}
\e_0 := \min \Big\{ 1 , \, 
\frac{1}{9} \Big( \frac{32}{3} \frac{\a_1^2 + \a_2^2 + \a_3^2}{\a_1^3} \Big)^{\frac12} \,, \,
\frac{\delta_1}{\sqrt{8} \a_4^{m_1}} \Big\},
\end{equation}
where $\d_1$ is the universal constant in \eqref{smallness.q2} 
and in \eqref{ball.mitica}, 
so that both \eqref{smallness.A1.q2.in.terms.of.eps} 
and \eqref{smallness.q2} are satisfied for all $0 < \e \leq \e_0$. 
The constant $\e_0$ depends only on $m,p$. 
For $\s = m/p$ small enough, the minimum in \eqref{def.eps.0} 
is the third element of the set, as we note in the following lemma.

\begin{lemma} \label{lemma:maybe.not.useful}
There exists a universal constant $\s_1 \in (0, 1)$ such that, 
if $\s = m/p \leq \s_1$, then $\e_0$ defined in \eqref{def.eps.0} 
is $\e_0 = \delta_1 / (\sqrt{8} \a_4^{m_1})$. 
\end{lemma}

\begin{proof} It is enough to recall that $\a_1^2 + \a_2^2 + \a_3^2 > 3 \a_1^2$, 
$\a_4^{m_1} \geq \a_4$, and $\a_4 > \a_1 > \sqrt{\a_1}$.
\end{proof}


\begin{lemma} \label{lemma:eps}
Let $a, m, p, \s, \a_1, \ldots \a_4, 
\tilde \mu_1(\s), \tilde \mu_2(\s), \g, 
(\xi_1(t), \eta_1(t), \xi_2(t), \eta_2(t))$ 
be like in Lemma \ref{lemma:back.from.xi.eta.to.6.eq}.
Consider any $\e \in (0,\e_0)$, 
where $\e_0$ is the constant, depending only on $m, p$, defined in \eqref{def.eps.0}, 
and define 
$q_2 := \e^2$, 
$\tA_1$  given by the second identity in \eqref{def.eps},
$q_1$ by \eqref{fix.ratio.q1.q2},
$E_1, E_2$ by \eqref{E12.from.q12},
and $c_{123}$, 
$c_{234}$, 
$\tB$, 
$\tA_2$, 
$S_1(t)$, 
$S_2(t)$,
$S_3(t)$, 
$S_4(t)$, 
$\ph_{123}(t)$,
$\ph_{234}(t)$
like in Lemma \ref{lemma:back.from.xi.eta.to.6.eq}.
Also define $\rho_{123} = \frac83 c_{123} / (\a_1 \a_2 \a_3)$, 
$\rho_{234} = \frac83 c_{234} / (\a_2 \a_3 \a_4)$. 

Then 
$(S_1(t), S_2(t), S_3(t), S_4(t), \ph_{123}(t), \ph_{234}(t))$ 
is a solution of system \eqref{syst.6.eq} for all $t \in \R$. 
Moreover, 
the identity \eqref{fix.ratio.q1.q2} is satisfied, 
and therefore \eqref{optimize.constant.terms.E.q} holds;
the inequality \eqref{smallness.A1.q2} is satisfied, 
and therefore the lower bound \eqref{bound.S.below} holds; 
the inequality \eqref{smallness.q2} is satisfied, 
and therefore the bound \eqref{ball.m1.S} holds; 
the inequality \eqref{smallness.A1.q2.power.32} is satisfied, 
and therefore the thesis of Lemma \ref{lemma:good.u0} holds.
\end{lemma}

\begin{proof} 
The proof follows from \eqref{def.eps}, \eqref{def.eps.0} 
and the results of the previous subsections.
\end{proof}


Recalling that the function $\eta_2$ 
given by Proposition \ref{prop:chaos.xi.eta} satisfies 
\eqref{bounds.eta2.in.section.6},
we also have the following lemma, where $\e^2, \e^3$ 
are isolated from the other parts of the coefficients. 
The reason to consider the quantity $\mN_1$ 
in \eqref{Sn.norm.1}
is that it corresponds to the square of the Sobolev norm
$H^1(\T^d)$ of the solution of the Kirchhoff equation.

To simplify the exposition of the lemma, we rewrite  \eqref{eq:systemmathtt} as
\begin{equation}\label{S.1234.eps}
	\begin{aligned}
		S_1(t) & = \e^2 s_1 - \e^3 a_1 \eta_1(b \e^3 t), \\
		S_2(t) & = \e^2 s_2 - \e^3 a_1 \eta_1(b \e^3 t) - \e^3 a_2 \eta_2(b \e^3 t),\\
		S_3(t) & = \e^2 s_3 + \e^3 a_1 \eta_1(b \e^3 t) - \e^3 a_2 \eta_2(b \e^3 t), \\
		S_4(t) & = \e^2 s_4 + \e^3 a_2 \eta_2(b \e^3 t),
	\end{aligned}
\end{equation}
where 
\begin{equation}\label{def.s.1234.a1.a2.b}
	\begin{gathered}
		s_1 := \frac{\a_4^2}{3 \a_1^2}, 
		\quad \quad 
		s_2 := \frac{\a_4^2}{3 \a_2^2}, 
		\quad \quad 
		s_3 := \frac{2 \a_4^2}{3 \a_3^2}, 
		\quad \quad 
		s_4 := 1,\\
		a_1 := \Big( \frac{3}{32} \frac{\a_1^3}{\a_1^2 + \a_2^2 + \a_3^2} \Big)^{\frac12},
		\quad \quad 
		a_2 := \frac{1}{\g} a_1,
		\quad \quad 
		b := \frac{\a_1^2 + \a_2^2 + \a_3^2}{2} a_1,
	\end{gathered}
\end{equation}
so that 
\begin{equation}  \label{A1.A2.B.a1.a2.b.eps}
	\tA_1 = a_1 \e^3, \quad \ 
	\tA_2 = a_2 \e^3, \quad \ 
	\tB = b \e^3.
\end{equation}
Note that the constants $s_1, s_2, s_3, s_4, a_1, a_2, b$ depend only on $m,p$.
We also define the associated function
\begin{align} 
	\mN_1(t) & := \sum_{n=1}^4 \a_n^2 S_n(t) 
	= \e^2 \frac{7 \a_4^2}{3} 
	+ \e^3 (2 \a_1 \a_2) a_1 \eta_1(b \e^3 t)
	+ \e^3 (2 \a_2 \a_3) a_2 \eta_2(b \e^3 t).
	\label{Sn.norm.1}
\end{align}

\begin{lemma} \label{lemma:Sn.eps}
%
There exists a universal constant $\s_2 \in (0,1)$ 
with the following property. 
Let $t_j, \bar t_j$ be given by Proposition \ref{prop:chaos.xi.eta}, 
and define 
\begin{equation} \label{def.rescaled.times}
t_j^* := \frac{t_j}{B} = \frac{t_j}{b \e^3}, 
\qquad 
\bar t_j^* := \frac{\bar t_j}{B} = \frac{\bar t_j}{b \e^3}, 
\qquad 
I_j^{*} := [t_j^*, \bar t_j^*], 
\qquad 
E_j^{*} := [\bar t_j^*, t_{j+1}^*].
\end{equation}
Assume the hypotheses of Lemma \ref{lemma:eps}. 
If, in addition, the ratio $\s = m/p$ satisfies $\s \leq \s_2$, then 
$\mN_1(t)$ defined in \eqref{Sn.norm.1} satisfies
\begin{align}
& \e^2 c_1 + \frac{99}{100} \e^3 r_1 \leq \mN_1(t) \leq \e^2 c_1 + \frac{201}{100} \e^3 r_1
\quad \forall t \in I_j^{*},
\notag \\
& \e^2 c_1 - \frac{1}{100} \e^3 r_1 \leq \mN_1(t) \leq \e^2 c_1 + \frac{101}{100} \e^3 r_1
\quad \forall t \in E_j^{*},
\notag \\
& \max_{t \in I_j^*} \mN_1(t) 
\geq \e^2 c_1 + \frac{199}{100} \e^3 r_1,
\quad \quad 
\min_{t \in E_j^*} \mN_1(t) 
\leq \e^2 c_1 + \frac{1}{100} \e^3 r_1,
\label{bounds.mE.1}
\end{align}
for all $j \geq 0$, where 
$c_1 := \frac73 \a_4^2$, 
$r_1 := (2 \a_2 \a_3) a_2$ depend only on $m,p$.
\end{lemma}


\begin{proof}
To prove the first inequality in the last line of \eqref{bounds.mE.1}, 
consider the interval $I_j^*$, 
and let $t_j'$ be a point in that interval 
where the function $\eta_2(b \e^3 t)$ achieves its maximum value. 
Since by \eqref{bounds.eta2.in.section.6} and \eqref{eta.12.bound}
\[
\eta_2(b\e^3 t_j') \geq 2 - C\sigma
\quad \text{and} \quad 
|\eta_1(t)| \leq 3 \quad \forall t \in \R,
\]
from \eqref{Sn.norm.1} 
and the identity $a_1 = \g a_2$ (see \eqref{def.s.1234.a1.a2.b})
we get 
\begin{align*}
\max_{t \in [t_j^*, \bar t_j^*]} \mN_1(t) 
\geq \mN_1(t_j') 
& \geq \e^2 \frac{7 \a_4^2}{3} 
+ \e^3 (2 \a_2 \a_3) a_2 (2-C\s)
- \e^3 (2 \a_1 \a_2) 3 a_1
\\
& = \e^2 \frac{7 \a_4^2}{3} 
+ \e^3 (2 \a_2 \a_3) a_2 (2-C\s) 
\Big( 1 - \frac{ 3 \gamma \a_1 }{ \a_3 (2- C\s) } \Big).
\end{align*}
The difference in the last parenthesis tends to 1 as $\s \to 0$, 
because $\g \to 3$ and the ratio
$\a_1 / \a_3 = m / (2m+p)$ tends to $0$ as $\s = m/p \to 0$. 
The other inequalities in \eqref{bounds.mE.1} are proved similarly, 
using \eqref{borraccia}, \eqref{bounds.eta2.in.section.6}, and the fact that
\begin{equation}\label{S4.N1}
\mN_1 (t) = \e^2 \frac{7 \a_4^2}{3} 
+ \e^3 (2 \a_2 \a_3) a_2 \Big( \eta_2(b\e^3 t) + \frac{\a_1 \gamma}{\a_3} \eta_1 (b\e^3 t) \Big)
= \e^2 c_1 + \e^3 r_1 \big( \eta_2 (b\e^3 t) + O(\s) \big). \ 
\end{equation}
\end{proof}

\begin{remark}
By \eqref{S4.N1}, the oscillations of $\mN_1$ at the main order in $\s$ are fully described by the ones of $\eta_2$, i.e., of $S_4$.
\end{remark}

\textbf{Notation}. 
From now on, we will sometimes be much less accurate than before 
in keeping track of the explicit dependence 
of the various constants on $\a_1, \a_2, \a_3, \a_4$; 
we will denote generically by $K$ (or sometimes $K'$, or $K''$)
any constant, possibly different from line to line, 
that depends only on the integers $m,p$.

\medskip

With the new notation, 
the constants defined in Lemma \ref{lemma:eps} become
\begin{equation}  \label{all.K.eps}
c_{123} = K \e^6, 
\quad \ 
c_{234} = K \e^6,
\quad \
\rho_{123} = K \e^6, 
\quad \ 
\rho_{234} = K \e^6,
\end{equation} 
where the four constants $K$ denote four (possibly different) values.

\section{Approximation argument}
\label{sec:gronwall}


Consider a solution 
$(S_1(t), S_2(t), S_3(t), S_4(t), \ph_{123}(t), \ph_{234}(t))$ 
of system \eqref{syst.6.eq} obtained in Lemma \ref{lemma:eps},
and let $u_0$ be the corresponding trigonometric polynomial 
in the $H^{m_1}$ ball \eqref{ball.mitica} 
constructed in Lemma \ref{lemma:good.u0}. 
Let $v_0 := \overline{u_0}$ be the complex conjugate of $u_0$,
and let $(u(t),v(t))$ be the solution of the Cauchy problem 
for the transformed Kirchhoff equation \eqref{3101.16} 
with initial condition \eqref{initial.cond.u.v.sec.prepar}.
Since $u_0$ is a trigonometric polynomial, 
the solution $(u,v)$ is global in time. 
Moreover, since $u_0$ belongs to the ball \eqref{ball.mitica},
the solution $(u,v)$ satisfies \eqref{twice.norm.sec.prepar} 
on the time interval $[0,T_{\mathrm{NF}}]$, see Lemma \ref{lemma:total}.

The estimate in Lemma \ref{lemma:ball.m1.S} computed at time $t=0$ 
gives the inequality 
\begin{equation}
\| u_0 \|_{s}^2 \leq 8 \a_4^{2 s} q_2 
\quad \forall s \geq 1,
\end{equation} 
and therefore, recalling \eqref{def.eps} 
and the notation at the end of the previous section,
$T_{\mathrm{NF}}$ in \eqref{twice.norm.sec.prepar} 
can be also estimated in terms of $\e$ as
\begin{equation} \label{T.NF.lower.bound}
T_{\mathrm{NF}} \geq \frac{C_0}{(8 \a_4^{2 m_1} q_2)^2} 
= K \e^{-4}.
\end{equation}
For each time $t \in [0, T_{\mathrm{NF}}]$, 
to slightly simplify the notations 
\eqref{def.SB.(u)},
\eqref{def.Z.rho.c.(u)},
\eqref{def.polar.(u)},
we denote
\begin{equation} \label{def.S.B.full}
S_n^u(t) 
:= S_n^{(u(t))}
\end{equation} 
the superactions of the function $u(t, \cdot)$,
and we also introduce the analogous notation 
for all the other quantities in \eqref{def.SB.(u)},
\eqref{def.Z.rho.c.(u)},
\eqref{def.polar.(u)}.

By Lemma \ref{lemma:Lemma.2.4.SIAM}, 
one has
\begin{equation} \label{S.t.S.0}
C_1 S_n^u(0) \leq S_n^u(t) \leq C_2 S_n^u(0) 
\quad \forall t \in [0, T_{\mathrm{NF}}], 
\quad n=1,2,3,4.
\end{equation}
By construction,
one has $S_n^u(0) = S_n(0)$. 
Hence, by \eqref{S.t.S.0}, \eqref{bound.norm.s.S}, \eqref{def.eps},
\begin{align} 
\| u(t) \|_s^2 
& = \sum_{n=1}^4 \a_n^{2s} S_n^u(t) 
\notag \\ 
& \leq C_2 \sum_{n=1}^4 \a_n^{2s} S_n^u(0)
= C_2 \sum_{n=1}^4 \a_n^{2s} S_n(0)
= C_2 \| u_0 \|_s^2 
\leq 8 C_2 \a_4^{2s} q_2
= 8 C_2 \a_4^{2s} \e^2
\label{bound.u(t).norm.s}
\end{align}
for all $t \in [0, T_{\mathrm{NF}}]$, 
all $s \geq 1$. 
For $s \in \{ 1, m_1 \}$, one has $\a_4^{2s} \leq \a_4^{2 m_1}$, 
and therefore we can simply write 
\begin{equation} \label{bound.u(t).norm.s.special}
\| u(t) \|_s \leq K \e 
\qquad \forall t \in [0, T_{\mathrm{NF}}], 
\quad s \in \{ 1, m_1 \}. 
\end{equation}

Since $(u(t), v(t))$ solves \eqref{3101.16} on $[0, T_{\mathrm{NF}}]$, 
the functions $S_n^u(t)$, $B_n^u(t)$ 
solve the effective equations \eqref{3105.8} 
on the same time interval, 
and $Z_{123}^u(t)$, $Z_{234}^u(t)$
solve \eqref{3105.11} on the same time interval. 
Since $\rho_{123}^u(t), \rho_{234}^u(t)$ are continuous 
and they are positive at time $t=0$, 
 there exists $T_{\mathrm{polar}}>0$ 
such that $\rho_{123}^u(t), \rho_{234}^u(t)$ are positive 
for all $t \in [0, T_{\mathrm{polar}}]$. 

\begin{remark}
Note that, in general, $T_{\mathrm{polar}}$ 
could be smaller than $T_{\mathrm{NF}}$. 
\end{remark}
On the time interval $[0, T_{\mathrm{polar}}]$, 
the angles $\ph_{123}^u(t), \ph_{234}^u(t)$ are well-defined 
by the identities
\begin{equation} \label{def.polar.u(t)}
Z_{123}^u (t) = \rho_{123}^u (t) \exp(i \ph_{123}^u (t) ),
\quad 
Z_{234}^u (t) = \rho_{234}^u (t) \exp(i \ph_{234}^u (t) ),
\quad 
\forall t \in [0, T_{\mathrm{polar}}].
\quad 
\end{equation}
Since $Z_{123}^u (t)$, $Z_{234}^u (t)$ solve \eqref{3105.11}, 
the functions $\rho_{123}^u (t)$, $\rho_{234}^u (t)$ solve the equations
\begin{equation} \label{eq.rho.234.u(t)}
\pa_t \rho_{123}^u = R_{\rho_{123}^u},
\qquad
\pa_t \rho_{234}^u = R_{\rho_{234}^u}
\end{equation}
on $[0, T_{\mathrm{polar}}]$, 
and the functions $\ph_{123}^u (t)$, $\ph_{234}^u (t)$ solve the equations
\begin{align} 
\pa_t \ph_{123}^u 
& = - \frac12 (\a_1^2 S_1^u + \a_2^2 S_2^u - \a_3^2 S_3^u) + R_{\ph_{123}^u},
\notag \\
\pa_t \ph_{234}^u 
& = - \frac12 (\a_2^2 S_2^u + \a_3^2 S_3^u - \a_4^2 S_4^u) + R_{\ph_{234}^u}
\label{eq.ph.234.u(t)}
\end{align}
on $[0, T_{\mathrm{polar}}]$, where
\begin{equation} \label{def.R.rho.ph.123}
R_{\rho_{123}^u} 
:= \Re \Big( \exp( - i \ph_{123}^u ) R_{Z_{123}^u} \Big),
\qquad 
R_{\ph_{123}^u} 
:= \frac{1}{\rho_{123}^u} 
\Im \Big( \exp( - i \ph_{123}^u ) R_{Z_{123}^u} \Big),
\end{equation}
and analogous definition for $R_{\rho_{234}^u}$, $R_{\ph_{234}^u}$ 
(just replace $123$ with $234$ everywhere in \eqref{def.R.rho.ph.123}).
The remainders of the type $R_Z$ are defined in \eqref{3105.12} 
and estimated in \eqref{0106.10}.
Moreover, $S_n^u$, $n=1,2,3,4$, solve 
\begin{align}
\pa_t S_1^u & = \frac38 \rho_{123}^u \a_1 \a_2 \a_3 \sin(\ph_{123}^u) + R_{S_1^u},
\notag \\
\pa_t S_2^u & = \frac38 \rho_{123}^u \a_1 \a_2 \a_3 \sin(\ph_{123}^u) 
+ \frac38 \rho_{234}^u \a_2 \a_3 \a_4 \sin(\ph_{234}^u) + R_{S_2^u},
\notag \\
\pa_t S_3^u & = - \frac38 \rho_{123}^u \a_1 \a_2 \a_3 \sin(\ph_{123}^u) 
+ \frac38 \rho_{234}^u \a_2 \a_3 \a_4 \sin(\ph_{234}^u) + R_{S_3^u},
\notag \\
\pa_t S_4^u & = - \frac38 \rho_{234}^u \a_2 \a_3 \a_4 \sin(\ph_{234}^u) + R_{S_4^u}
\label{eq.S.u}
\end{align}
on $[0, T_{\mathrm{polar}}]$, where the remainders of the type $R_S$ 
appear in \eqref{3105.8} and are estimated in \eqref{3005.3}.
In the following lemma we prove a formula for $T_{\mathrm{polar}}$.

\begin{lemma} \label{lemma:T.polar}
Assume the hypotheses of Lemma \ref{lemma:eps}. 
There exists a universal constant $\s_3 \in (0,1)$ such that, 
if $\s= m/p$, in addition to the hypotheses of Lemma \ref{lemma:eps},  
also satisfies 
$\s  \leq \s_3$, 
then $Z_{123}^u(t) \neq 0$, $Z_{234}^u(t) \neq 0$
for all $t \in [0, T_{\mathrm{polar}}]$, with 
\begin{equation} \label{formula.T.polar.short}
T_{\mathrm{polar}} = \frac{ C \a_1 A_1^2 }{ \a_4^{4m_1 +2} q_2^5 } = K \e^{-4}
\leq T_{\mathrm{NF}},
\end{equation}
where $C>0$ is a universal constant
and $K > 0$ is a constant depending only on $m,p$. 
As a consequence, on $[0, T_{\mathrm{polar}}]$, $\rho_{123}^u, \rho_{234}^u$ are positive and
$\ph_{123}^u$, $\ph_{234}^u$ are well-defined . Moreover, 
\begin{equation} \label{est.rho.gross}
|\rho_{123}^u(t) - \rho_{123}| \leq \frac12 \rho_{123}, 
\qquad 
|\rho_{234}^u(t) - \rho_{234}| \leq \frac12 \rho_{234}.
\end{equation}
\end{lemma}

\begin{proof}
In this proof $C$ (and $C'$, $C''$) denote universal constants, 
possibly different from line to line. 
Suppose that $\rho_{123}^u$ is positive on some 
time interval $[0,T] \subset [0, T_{\mathrm{NF}}]$.
Then, by \eqref{eq.rho.234.u(t)} and \eqref{def.R.rho.ph.123}, 
\[
|\pa_t \rho_{123}^u| 
= \big| \Re \big( \exp( - i \ph_{123}^u ) R_{Z_{123}^u} \big) \big|
\leq \big| \exp( - i \ph_{123}^u ) R_{Z_{123}^u} \big|
= |R_{Z_{123}^u}|.
\]
By \eqref{0106.10}, for all $s \in \R$,
\[
|R_{Z_{123}^u}| \leq C \| u \|_{m_1}^4 S_1^u S_2^u S_3^u
= \frac{C \| u \|_{m_1}^4}{(\a_1 \a_2 \a_3)^{2s}} 
(\a_1^{2s} S_1^u) (\a_2^{2s} S_2^u) (\a_3^{2s} S_3^u)
\leq \frac{ C \| u \|_{m_1}^4 \| u \|_s^6 }{(\a_1 \a_2 \a_3)^{2s}}
\]
and, by \eqref{bound.u(t).norm.s}, for  $s \geq 1$ and $t \in [0,T]$,
\[
\| u(t) \|_{m_1}^4 \| u(t) \|_s^6 
\leq C \a_4^{4m_1+6s} q_2^5.
\]
Hence, since $\a_4 \leq 2 \a_3$ and $\a_4 \leq 3 \a_2$, for $s=1$, we obtain
\begin{equation} \label{est.pat.rho.123}
|\pa_t \rho_{123}^u| 
\leq |R_{Z_{123}^u}|
\leq \frac{ C \a_4^{4m_1+6} q_2^5 }{(\a_1 \a_2 \a_3)^{2}}
\leq \frac{ C' \a_4^{4m_1+2} q_2^5 }{ \a_1^2 }
\quad \forall t \in [0, T].
\end{equation}
At time $t=0$ we have, by construction, 
and by the definition \eqref{c123.from.A1} of $c_{123}$,
\begin{equation} \label{recall.rho.123}
\rho_{123}^u(0) = \rho_{123} 
= \frac{8 c_{123} }{3 \a_1 \a_2 \a_3}
= \frac{ 4 (\a_1^2 + \a_2^2 + \a_3^2) }{ 3 \a_1 \a_2 \a_3 } \tA_1^2,
\end{equation}
which is positive because $\tA_1 > 0$. 
Hence 
\[
|\rho_{123}^u(t) - \rho_{123}| 
\leq \int_0^t |\pa_t \rho_{123}^u(s)| \, ds
\leq \frac{ C' \a_4^{4m_1+2} q_2^5 }{ \a_1^2 } t 
\quad \forall t \in [0,T],
\]
and the last quantity is strictly less than $\rho_{123}$ for all $t \in [0,T]$ if 
\begin{equation} \label{T.sup.123}
T < \frac{ C \a_1 }{ \a_4^{4m_1 +2}} \frac{ \tA_1^2 }{ q_2^5 } =: T_{123}^*.
\end{equation}
This implies that $\rho_{123}^u > 0$ on $[0,T_{123}^*)$. 
Moreover, by \eqref{smallness.A1.q2.power.32}  
and \eqref{T.NF.lower.bound}, 
\[
T_{123}^* = \frac{ C \a_1 }{ \a_4^{4m_1 +2}} \frac{ \tA_1^2 }{ q_2^5 } 
\leq C' \frac{\a_1^4}{\a_4^4} \frac{1}{\a_4^{2m_1} q_2^2}
\leq C'' \s^4 T_{\mathrm{NF}},
\]
where $\s = m/p$. 
Then, for $C'' \s^4 \leq 1$, 
one has $T_{123}^* \leq T_{\mathrm{NF}}$.

Proceeding similarly for $\rho_{234}^u$, we obtain the estimate  
\begin{equation} \label{est.pat.rho.234}
|\pa_t \rho_{234}^u| 
\leq |R_{Z_{234}^u}|
\leq \frac{ C \a_4^{4m_1+6} q_2^5 }{(\a_2 \a_3 \a_4)^{2}}
\leq C' \a_4^{4m_1} q_2^5,
\end{equation}
and we deduce that $\rho_{234}^u > 0$ on $[0, T_{234}^*)$, with 
\[
T_{234}^* = \frac{C \mathtt{A}_1^2}{\a_4^{4m_1+1} q_2^5} 
\leq C' \s^3 T_{\mathrm{NF}}.
\]
Also, $T_{123}^* \leq T_{234}^* C \s \leq T_{234}^*$ for $C \s \leq 1$. 

Finally, the estimates already proved also give that 
$|\rho_{123}^u(t) - \rho_{123}| \leq \frac12 \rho_{123}$ 
for all $t \in [0, T_{123}^*/2]$. 
Thus, we fix $T_{\mathrm{polar}} = T_{123}^*/2$.
\end{proof}

We recall that the superscript $u$ indicates quantities 
related to the solution $(u,v)$ of the Cauchy problem 
\eqref{3101.16}, \eqref{initial.cond.u.v.sec.prepar}, 
while the absence of that superscript corresponds 
to the solution of the truncated effective system \eqref{syst.6.eq}.
By construction, we have 
\[
S^u(0) = S(0), \qquad 
\rho^u(0) = \rho(0), \qquad 
\ph^u(0) = \ph(0).
\]
Moreover, $\rho$ is constant in time, 
and there are no remainders $R_S, R_\rho, R_\ph$ without the superscript $u$
because \eqref{syst.6.eq} is the truncated effective system.

The next lemma gives estimates for the difference between $S^u(t)$ and $S(t)$ in a certain time interval.

\begin{lemma} \label{lemma:after.Gronwall}
Assume the hypotheses of Lemma \ref{lemma:eps},
and consider the associated solution $(S_1(t),$ $S_2(t), S_3(t), S_4(t), \ph_{123}(t),\ph_{234}(t))$ 
of system \eqref{syst.6.eq} given by Lemma \ref{lemma:eps}.
Let $u_0$ be the corresponding trigonometric polynomial in the ball \eqref{ball.mitica} 
constructed in Lemma \ref{lemma:good.u0}. 
Let $v_0 := \overline{u_0}$ be the complex conjugate of $u_0$,
and let $(u(t),v(t))$ be the solution of the Cauchy problem 
for the transformed Kirchhoff equation \eqref{3101.16} 
with initial condition \eqref{initial.cond.u.v.sec.prepar}. 
Assume that the ratio $\s = m/p$  satisfies 
$\s \leq \s_3$, 
where $\s_3$ is defined in Lemma \ref{lemma:T.polar}.
	
There exists $\e_1 \in (0, \e_0)$, depending only on $m,p$,
where $\e_0$ is the constant in \eqref{def.eps.0}, 
such that, for $0 < \e \leq \e_1$, one has
	\begin{equation}  \label{bounds.Sn.Gronwall}
		\e^{-3} |S_n^u(t) - S_n(t)| \leq \e^{\frac34}
		\quad \ \forall t \in [0, T_\e], 
		\quad n=1,2,3,4,
	\end{equation}
where 
\begin{equation}  \label{formula.T.eps}
	T_\e = K \e^{-3} \log(\e^{-1})
\end{equation}
for some  positive constant $K$ depending only on $m,p$.  
\end{lemma}

\begin{proof}
From the difference of equations \eqref{eq.S.u} for $S_n^u$ 
and equations \eqref{syst.6.eq} for $S_n$,
we obtain that for all $t \in [0, T_{\mathrm{polar}}]$, 
$n = 1,2,3,4$,
\begin{align*} 
| \pa_t (S_n^u - S_n)| 
& \leq \a_1 \a_2 \a_3 
\Big( \rho_{123} |\ph_{123}^u - \ph_{123}| + |\rho_{123}^u - \rho_{123}| \Big) 
\\ & \quad \ 
+ \a_2 \a_3 \a_4
\Big( \rho_{234} |\ph_{234}^u - \ph_{234}| + |\rho_{234}^u - \rho_{234}| \Big) 
+ |R_{S_n^u}|.
\end{align*}
By \eqref{3005.3} one has 
$|R_{S_n^u}| \leq C \| u \|_{m_1}^6 S_n^u$.
By \eqref{bound.u(t).norm.s.special}, 
\[
\| u \|_{m_1} \leq K \e, \qquad 
|S_n^u| \leq \a_n |S_n^u|
\leq \| u \|_1^2 
\leq K \e^2, 
\]
and therefore 
$|R_{S_n^u}| \leq K \e^8$  
on $[0, T_{\mathrm{polar}}]$. 
Hence, for $t \in [0, T_{\mathrm{polar}}]$ and $n=1,2,3,4$,
\begin{align} 
| \pa_t (S_n^u - S_n)| 
& \leq K \Big( \rho_{123} |\ph_{123}^u - \ph_{123}| + |\rho_{123}^u - \rho_{123}| \Big) 
\notag \\ & \quad \ 
+ K \Big( \rho_{234} |\ph_{234}^u - \ph_{234}| + |\rho_{234}^u - \rho_{234}| \Big) 
+ K \e^8.
\label{est.pat.S.temp}
\end{align}
Since $\rho_{123}, \rho_{234}$ are constants, 
by the equations \eqref{eq.rho.234.u(t)} for $\rho_{123}^u$, $\rho_{234}^u$ 
one has
\[
\pa_t(\rho_{123}^u - \rho_{123}) 
= \pa_t \rho_{123}^u 
= R_{\rho_{123}^u},
\qquad 
\pa_t(\rho_{234}^u - \rho_{234}) 
= \pa_t \rho_{234}^u
= R_{\rho_{234}^u}.
\]
The time derivatives $\pa_t \rho_{123}^u$, $\pa_t \rho_{234}^u$
have been already estimated in \eqref{est.pat.rho.123} and \eqref{est.pat.rho.234}; 
hence, integrating in time, we get
\[
|\rho_{123}^u(t) - \rho_{123}| 
\leq K \e^{10} t, 
\quad \ 
|\rho_{234}^u(t) - \rho_{234}| 
\leq K \e^{10} t
\quad \ 
\forall t \in [0, T_{\mathrm{polar}}].
\]
Also, by \eqref{recall.rho.123} and \eqref{A1.A2.B.a1.a2.b.eps},
$\rho_{123} \leq K \e^6$, 
$\rho_{234} \leq K \e^6$. 
We plug these estimates into \eqref{est.pat.S.temp} and we get
\begin{align} 
| \pa_t (S_n^u - S_n)| 
& \leq K \e^6 \big( |\ph_{123}^u - \ph_{123}| + |\ph_{234}^u - \ph_{234}| \big)
+ K \e^{10} t + K \e^8
\label{est.pat.S.temp.2}
\end{align}
for all $t \in [0, T_{\mathrm{polar}}]$, $n=1,2,3,4$.

Subtracting the equations \eqref{eq.ph.234.u(t)} 
for $\ph_{123}^u$, $\ph_{234}^u$ 
and those for $\ph_{123}, \ph_{234}$ in \eqref{syst.6.eq},
we have
\begin{align*}
|\pa_t (\ph_{123}^u - \ph_{123})| 
& \leq \a_1^2 |S_1^u - S_1| + \a_2^2 |S_2^u - S_2| 
+ \a_3^2 |S_3^u - S_3| + |R_{\ph_{123}^u}|,
\\
|\pa_t (\ph_{234}^u - \ph_{234})| 
& \leq \a_2^2 |S_2^u - S_2| + \a_3^2 |S_3^u - S_3| 
+ \a_4^2 |S_4^u - S_4| + |R_{\ph_{234}^u}|.
\end{align*}
By \eqref{est.rho.gross}, $\rho_{123}^u \geq \frac12 \rho_{123}$ 
on $[0, T_{\mathrm{polar}}]$. 
Hence, by \eqref{def.R.rho.ph.123},
$|R_{\ph_{123}^u}| \leq (2 / \rho_{123}) |R_{Z_{123}^u}|$, 
and $R_{Z_{123}^u}$ has been estimated in \eqref{est.pat.rho.123}.
Also, $\rho_{123}$ is given in \eqref{recall.rho.123}. Therefore
\[
|R_{\ph_{123}^u}|
\leq \frac{2}{\rho_{123}} |R_{Z_{123}^u}|
\leq \frac{K q_2^5}{\mathtt{A}_1^2} \, 
= K \e^4
\]
and, similarly, $|R_{\ph_{234}^u}| \leq K \e^4$. 
Thus
\begin{equation} \label{est.pat.ph.temp}
|\pa_t (\ph_{123}^u - \ph_{123})| 
+ |\pa_t (\ph_{234}^u - \ph_{234})| 
 \leq K \sum_{n=1}^4 |S_n^u - S_n| 
+ K \e^4. 
\end{equation}

We apply  Gronwall's inequality to \eqref{est.pat.S.temp.2} and \eqref{est.pat.ph.temp}. 
It is convenient to introduce a factor $\e^\b$, 
with $\b$ to be determined, 
because the factors in \eqref{est.pat.S.temp.2} and \eqref{est.pat.ph.temp} 
contain different powers of $\e$.
We define the vector 
\begin{align}
\psi & 
:= \big( 
\e^\b (S_1^u - S_1), 
\e^\b (S_2^u - S_2), 
\e^\b (S_3^u - S_3), 
\e^\b (S_4^u - S_4), 
(\ph_{123}^u - \ph_{123}), 
(\ph_{234}^u - \ph_{234}) \big)
\notag 
\\
& =: (\psi_1, \psi_2, \psi_3, \psi_4, \psi_5, \psi_6),
\label{def.psi}
\end{align}
which is a function of $t \in [0, T_{\mathrm{polar}}]$ 
taking values in $\R^6$. 
From \eqref{est.pat.S.temp.2} we obtain 
\begin{align*} 
|\pa_t \psi_n| = \e^\b |\pa_t (S_n^u - S_n)| 
& \leq \e^\b [ K \e^6 (|\psi_5| + |\psi_6|) + K \e^{10} t + K \e^8 ] 
\\ 
& \leq K \e^{6+\b} |\psi| + K \e^{10 + \b} t + K \e^{8 + \b},
\quad \ n = 1,2,3,4,
\end{align*}
and from \eqref{est.pat.ph.temp} we get
\begin{align*}
|\pa_t \psi_5| + |\pa_t \psi_6|    
& \leq K \e^{-\b} (|\psi_1| + |\psi_2| + |\psi_3| + |\psi_4|) + K \e^4
\leq K \e^{-\b} |\psi| + K \e^4,
\end{align*}
where $| \ |$ is the usual Euclidean norm of $\R^6$. 
Therefore 
\begin{equation} \label{before.Gronwall}
|\pa_t \psi| 
\leq K (\e^{6+\b} + \e^{-\b}) |\psi| 
+ K \e^{10 + \b} t + K \e^{8 + \b} + K \e^4.
\end{equation}
We fix the value of $\b$ that minimizes the sum $\e^{6+\b} + \e^{-\b}$, 
that is, $\b = -3$. 
Hence \eqref{before.Gronwall} becomes
\begin{equation} \label{before.Gronwall.beta.fixed}
|\pa_t \psi| 
\leq K \e^3 |\psi| 
+ K \e^7 t + K \e^4.
\end{equation}
Moreover, by the choice of the initial conditions $\psi(0) = 0$. 
Then, by Gronwall's inequality,
\begin{equation} \label{psi.after.Gronwall}
|\psi(t)| \leq \exp ( K_0 \e^3 t ) K_1 ( \e^7 t^2 + \e^4 t )
\quad \ 
\forall t \in [0, T_{\mathrm{polar}}],
\end{equation}
for some positive constants $K_0, K_1$ depending only on $m,p$.

Now we consider the time 
$T_{\e,h} := K_0^{-1} \e^{-3} \log( \e^{-h} )$,
where $h$ is any positive real constant and $K_0>0$ is the constant appearing in \eqref{psi.after.Gronwall}; 
we note that $T_{\e,h}$ depends on $h, \e, m, p$. 
By \eqref{formula.T.polar.short},
one has 
$T_{\e,h} \leq T_{\mathrm{polar}}$ for
$0 < \e \leq C$,
for some positive constant $C$ depending only on $h,m,p$. 
By \eqref{psi.after.Gronwall}, for all $t \in [0, T_{\e,h}]$ one has 
\begin{align*}
|\psi(t)| 
& \leq \exp ( K_0 \e^3 t ) K_1 ( \e^7 t^2 + \e^4 t )
\leq \e^{-h} K_1 \Big( \e^{7} K_0^{-2} \e^{-6} \log^2( \e^{-h} )
+ \e^4 K_0^{-1} \e^{-3} \log( \e^{-h} ) \Big)
\\
& \leq \e^{1-h} K_1 \Big( K_0^{-2} \log^2( \e^{-h} )
+ K_0^{-1} \log( \e^{-h} ) \Big).
\end{align*}
We also note that 
\begin{equation} \label{replace.with.power}
\e^{1-h} K_1 \Big( K_0^{-2} \log^2( \e^{-h} ) + K_0^{-1} \log( \e^{-h} ) \Big)
\leq \e^{1 - 2h} 
\quad \ 
\text{for } 0 < \e \leq C',
\end{equation}
for some positive constant $C'$
depending on $h,m,p$. 
We can fix, for example, $h = 1/8$. 
\end{proof}

\subsection{Motion of the Sobolev norms} 

We use the $H^1(\T^d)$ Sobolev norm $\| u(t) \|_1$ 
to describe the transfer of energy, 
namely the exchanges in amplitude of the superactions 
$S_n^u(t)$, $n=1,2,3,4$, as time evolves. 
Except the $H^{\frac12}$ norm, 
which is constant in time for the solutions 
of the approximating system \eqref{syst.xi.eta.fully.normalized2}
given by Lemma \ref{lemma:back.from.xi.eta.to.6.eq} (see Remark \ref{norm.12.prime.int}), 
any other $H^s$ norms could be used to capture the transfer of energy 
among the superactions. 
We decide to use the $H^1$ norm 
because $H^1$ here corresponds to the space 
$H^{\frac32} \times H^{\frac12}$ for the ``physical variables'', 
which is the space of the standard local wellposedness for the Kirchhoff equation.

\begin{lemma} \label{lemma:chaos.u}
Assume all the hypotheses of Lemma \ref{lemma:after.Gronwall}, 
and also let $\s \leq \s_2$, where $\s_2$ is given by Lemma \ref{lemma:Sn.eps}. 
Let $t_j^*, \bar t_j^*, I_j^*, E_j^*$ be defined in \eqref{def.rescaled.times}. 
There exists a constant $\e_2 \in (0,\e_1)$, depending only on $m,p$, where $\e_1$ is defined in Lemma \ref{lemma:after.Gronwall},
such that if, in addition to the hypotheses of Lemma \ref{lemma:after.Gronwall}, 
$0<\e \leq \e_2$, 
then the $H^1(\T^d)$ Sobolev norm of $u(t)$ satisfies
\begin{align}
& \e^2 c_1 + \frac{98}{100} \e^3 r_1 \leq \| u(t) \|_1^2 
\leq \e^2 c_1 + \frac{202}{100} \e^3 r_1
\quad \ \forall t\in I_j^*,
\notag \\
& \e^2 c_1 - \frac{2}{100} \e^3 r_1 \leq \| u(t) \|_1^2 
\leq \e^2 c_1 + \frac{102}{100} \e^3 r_1
\quad \ \forall t\in E_j^*,
\notag \\
& \max_{t \in I_j^*} \| u(t) \|_1^2 
 \geq \e^2 c_1 + \frac{198}{100} \e^3 r_1,
\quad  \quad 
\min_{t \in E_j^*} \| u(t) \|_1^2
 \leq \e^2 c_1 + \frac{2}{100} \e^3 r_1,
\label{bounds.u.Sob.norm.1}
\end{align}
for all the indices $j \geq 0$ for which 
the intervals $I_j^*, E_j^*$ are contained in $[0,T_\e]$, where $T_\e$ is defined in \eqref{formula.T.eps} and $c_1,r_1$ are defined in Lemma \ref{lemma:Sn.eps}. 

Moreover, there exists a constant $K>0$, depending only on $m,p$, 
such that, if the integers $m_j$ introduced in Proposition \ref{prop:chaos.xi.eta} satisfy
\begin{equation} \label{cond.N}
\sum_{j=0}^N m_j \leq K \log(\e^{-1})
\end{equation}
for some $N$, 
then the intervals $I_0^*, E_0^*, \ldots, I_N^*, E_N^*$
are all contained in $[0,T_\e]$.
\end{lemma}

\begin{proof} 
Consider an interval $I_j^* \subseteq [0,T_\e]$, 
and a point $t_j'$ in that interval 
where the function $\mN_1$, defined in \eqref{Sn.norm.1},
achieves its maximum value over $I_j^*$. 
By \eqref{bounds.mE.1} and \eqref{bounds.Sn.Gronwall},  
\begin{align*}
\max_{t \in I_j^*} \| u(t) \|_1^2 
& \geq \| u(t_j') \|_1^2 
= \sum_{n=1}^4 \a_n^2 S_n^u(t_j')
= \mN_1(t_j') + \sum_{n=1}^4 \a_n^2 \big( S_n^u(t_j') - S_n(t_j') \big)
\\
& \geq \e^2 c_1 + \frac{199}{100} \e^3 r_1 - \sum_{n=1}^4 \a_n^2 \e^{3 + \frac34}
\geq \e^2 c_1 + \frac{198}{100} \e^3 r_1.
\end{align*}
The other inequalities in \eqref{bounds.u.Sob.norm.1} are proved similarly.

From \eqref{def:tj} one has 
$t_{j+1} = \tau (m_0 + \theta_0 + \ldots + m_j + \theta_j)$. 
Hence, by \eqref{def.rescaled.times}, 
\[
t_{j+1}^* = \frac{\tau}{b \e^3} \sum_{k=0}^j (m_k + \theta_k) 
\quad \ \forall j = 0,1,2,\ldots.
\]
Since $(m_k + \theta_k) \leq 2 m_k$,
recalling \eqref{formula.T.eps},  
one has $t_{N+1}^* \leq T_\e$ if \eqref{cond.N} holds. 
\end{proof}

\begin{lemma}  \label{lemma:chaos.u.without.square}
Assume the hypotheses of Lemma \ref{lemma:chaos.u}.  
There exists a constant $\e_3 \in (0,\e_2)$, 
depending only on $m,p$, where $\e_2$ is defined in Lemma \ref{lemma:chaos.u}, such that, 
if $0< \e \leq \e_3$, then
\begin{align}
& \e \tilde c_1 + \frac{48}{100} \e^2 \tilde r_1 \leq \| u(t) \|_1
\leq \e \tilde c_1 + \frac{101}{100} \e^2 \tilde r_1
\quad \ \forall t\in I_j^*,
\notag \\
& \e \tilde c_1 - \frac{2}{100} \e^2 \tilde r_1 \leq \| u(t) \|_1
\leq \e \tilde c_1 + \frac{51}{100} \e^2 \tilde r_1
\quad \ \forall t\in E_j^*,
\notag \\
& \max_{t \in I_j^*} \| u(t) \|_1 
\geq \e \tilde c_1 + \frac{98}{100} \e^2 \tilde r_1,
\qquad 
\min_{t \in E_j^*} \| u(t) \|_1
\leq \e \tilde c_1 + \frac{1}{100} \e^2 \tilde r_1,
 \label{bounds.u.Sob.norm.1.without.square} 
\end{align}
where 
$\tilde c_1 := \sqrt{c_1}$ and 
$\tilde r_1 := r_1 / \sqrt{ c_1 }$.
The inequalities in \eqref{bounds.u.Sob.norm.1.without.square} 
hold for the indices $j$ described in Lemma \ref{lemma:chaos.u}, 
that is, for $j = 0, \ldots, N$, where $N$ satisfies \eqref{cond.N}.
\end{lemma}

\begin{proof}
The inequalities in \eqref{bounds.u.Sob.norm.1.without.square} 
are obtained from \eqref{bounds.u.Sob.norm.1} 
by the Taylor expansion $(1 + x)^{1/2} = 1 + \frac12 x + O(x^2)$ as $x \to 0$ 
and the inequality $(1 + x)^{1/2} \leq 1 + \frac12 x$, 
which holds for all $x \geq -1$. 
\end{proof}

\subsection{Back to the solutions of the Kirchhoff equation}

As is observed in Lemma \ref{lemma:total}, 
if $(u_0, v_0)$ is in the ball \eqref{ball.mitica}, 
then, for all $t \in [0, T_{\mathrm{NF}}]$, 
the solution $(u,v)$ of the Cauchy problem 
\eqref{3101.16}, \eqref{initial.cond.u.v.sec.prepar}  
remains in the ball $\| u \|_{m_1} \leq \d$ 
where the transformation $\Phi$ is well-defined,
and $(\tilde u, \tilde v) = \Phi(u,v)$ in \eqref{tilde.uv.from.uv}
solves the original system \eqref{p1} on the same time interval. 

We want to prove that the solution $(\tilde u, \tilde v)$ 
has a dynamical behavior similar to the one of $(u,v)$ 
in Lemma \ref{lemma:chaos.u.without.square}.  
We underline that $v = \overline{u}$, 
hence $\| v \|_1 = \| u \|_1$, 
and, in fact, the inequalities in Lemma \ref{lemma:chaos.u.without.square} 
regard the solution $(u,v)$. 
On the contrary, the ``physical'' solution $(\tilde u, \tilde v)$ 
is a pair of real-valued functions solving \eqref{p1}, 
and therefore $\tilde v = \pa_t \tilde u$.
Thus, $\| u \|_1$ appearing in Lemma \ref{lemma:chaos.u.without.square} 
corresponds to $\mN$ in Lemma \ref{lemma:bound.phys.var}.

The transformation $\Phi$ is defined in \eqref{def.Phi}. 
We consider the map $\Phi_3 \circ \Phi_4 \circ \Phi_5$ first,
and then $\Phi_1, \Phi_2$.
From Lemma \ref{lemma:Lemma.2.9.SIAM} we deduce the following property. 




\begin{lemma} \label{lemma:cubic.bound.s}
Let $\d, C$ be the universal constants in Lemma \ref{lemma:Lemma.2.9.SIAM}. 
Let $(u,v) \in H^{m_1}_0(\T^d, c.c.)$, with $\| u \|_{m_1} \leq \d$, 
and let $(f,g) := \Phi_3 \circ \Phi_4 \circ \Phi_5 (u,v)$. 
Then, for every $s \in \R$,  
\begin{equation}  \label{cubic.bound.s}
\| f - u \|_s \leq 2 C \| u \|_{m_1}^2 \| u \|_s.
\end{equation}
\end{lemma}

\begin{proof} 
By \eqref{2.77.SIAM}, 
\begin{align*}
\| f - u \|_s^2 
= \sum_{k \in \Z^d} |k|^{2s} |f_k - u_k|^2 
& \leq \sum_k |k|^{2s} C^2 \| u \|_{m_1}^4 (|u_k| + |u_{-k}|)^2
\\
& \leq 2 C^2 \| u \|_{m_1}^4 \sum_k |k|^{2s}  (|u_k|^2 + |u_{-k}|^2)
= 4 C^2 \| u \|_{m_1}^4 \| u \|_s^2. 
\qedhere
\end{align*}
\end{proof}

We apply estimate \eqref{cubic.bound.s} 
to the solution $(u,v)$ of \eqref{3101.16}
constructed in the previous sections. 

\begin{lemma}  \label{lemma:bound.f.g}
Assume the hypotheses of Lemma \ref{lemma:chaos.u.without.square}. 
Let $v$ be the complex conjugate of $u$,  
and let $(f,g) := \Phi_3 \circ \Phi_4 \circ \Phi_5(u,v)$. 
There exists a constant $\e_4 \in (0,\e_3)$, 
depending only on $m,p$,
where $\e_3$ is given by Lemma \ref{lemma:chaos.u.without.square},
such that, if $0<\e \leq \e_4$, then 
\begin{align}
& \e \tilde c_1 + \frac{47}{100} \e^2 \tilde r_1 \leq \| f(t) \|_1
\leq \e \tilde c_1 + \frac{102}{100} \e^2 \tilde r_1
\quad \ \forall t\in I_j^*,
\notag \\
& \e \tilde c_1 - \frac{3}{100} \e^2 \tilde r_1 \leq \| f(t) \|_1
\leq \e \tilde c_1 + \frac{52}{100} \e^2 \tilde r_1
\quad \ \forall t\in E_j^*,
\notag \\
& \max_{t \in I_j^*} \| f(t) \|_1 
\geq \e \tilde c_1 + \frac{97}{100} \e^2 \tilde r_1,
\qquad 
\min_{t \in E_j^*} \| f(t) \|_1
\leq \e \tilde c_1 + \frac{2}{100} \e^2 \tilde r_1,
 \label{bounds.f} 
\end{align}
where $\tilde c_1$, $\tilde r_1$ are defined 
in Lemma \ref{lemma:chaos.u.without.square}.
The inequalities \eqref{bounds.f} hold 
for the indices $j$ described in Lemma \ref{lemma:chaos.u}, 
that is, for $j = 0, \ldots, N$, where $N$ satisfies \eqref{cond.N}.
\end{lemma}

\begin{proof}
By \eqref{cubic.bound.s} and \eqref{bound.u(t).norm.s.special}, 
one has
\begin{equation} \label{cubic.bounds.eps}
\| f - u \|_1 
\leq 2 C \| u \|_{m_1}^2 \| u \|_1
\leq K \e^3
\end{equation}
on the time interval $[0, T_{\mathrm{NF}}]$, 
where the constant $K$ depends only on $m,p$. 
All the inequalities in the lemma are proved by 
\eqref{bounds.u.Sob.norm.1.without.square} 
and \eqref{cubic.bounds.eps}.
\end{proof}

The transformations $\Phi_1, \Phi_2$ are simply these:
\begin{equation}  \label{f.g.q.p.tilde.u.tilde.v}
(q,p) = \Phi_2(f,g) = \Big( \frac{f+g}{\sqrt{2}}, \, \frac{f-g}{i\sqrt{2}} \Big), 
\qquad 
(\tilde u, \tilde v) = \Phi_1(q,p) = (|D_x|^{-\frac12} q, |D_x|^{\frac12} p),
\end{equation}
where $|D_x|^s$ is the Fourier multiplier 
$|D_x|^s e^{ik \cdot x} = |k|^s e^{ik \cdot x}$, $s \in \R$, 
$k \in \Z^d \setminus \{ 0 \}$ 
(the frequency $k=0$ can be ignored here,
because only zero-average functions are involved).

\begin{lemma} \label{lemma:isometry}
Let $s \in \R$, let $f$ be a zero-average, complex-valued function in $H^s(\T^d,\C)$, 
and let $g$ be its complex conjugate, i.e., $(f,g) \in H^s_0(\T^d,c.c.)$. 
Then $(q,p)$ defined in \eqref{f.g.q.p.tilde.u.tilde.v} is a pair of zero-average, 
real-valued functions in $H^s_0(\T^d,\R)$, with 
\[
\| q \|_s^2 + \| p \|_s^2 
= \| f \|_s^2 + \| g \|_s^2 
= 2 \| f \|_s^2,
\]
and $(\tilde u, \tilde v)$ defined in \eqref{f.g.q.p.tilde.u.tilde.v} 
is a pair of zero-average, real-valued functions 
in $H^{s+\frac12}_0(\T^d,\R) \times H^{s-\frac12}_0(\T^d,\R)$, 
with 
\[
\| \tilde u \|_{s+\frac12} = \| q \|_s, 
\qquad 
\| \tilde v \|_{s-\frac12} = \| p \|_s.
\]
\end{lemma}

\begin{proof} 
Elementary calculations with Fourier coefficients.
\end{proof}

The next lemma regards the solutions $(\tilde u, \tilde v)$ 
of system \eqref{p1}, that is, the solutions $\tilde u$ 
of the original Kirchhoff equation \eqref{K Om}.

\begin{lemma}  \label{lemma:bound.phys.var}
Assume the hypotheses of Lemma \ref{lemma:bound.f.g},
and let $(\tilde u, \tilde v) := \Phi_1 \circ \Phi_2(f,g)$. 
If $\e \leq \e_4$, then the function
\begin{equation}\label{def.calN}
\mN (t) := 
\Big( \| \tilde u (t) \|^2_{\frac32} + \| \tilde v (t) \|^2_{\frac12} 
\Big)^{\frac12}
\end{equation}
satisfies
\begin{align}
& \e c_0 + \frac{47}{100} \e^2 r_0 \leq \mN(t)
\leq \e c_0 + \frac{102}{100} \e^2  r_0
\quad \ \forall t\in I_j^*,
\notag \\
& \e c_0 - \frac{3}{100} \e^2 r_0 \leq \mN(t)
\leq \e c_0 + \frac{52}{100} \e^2 r_0
\quad \ \forall t\in E_j^*,
\notag \\
& \max_{t \in I_j^*} \mN(t)
\geq \e c_0 + \frac{97}{100} \e^2 r_0,
\qquad 
\min_{t \in E_j^*} \mN(t)
\leq \e c_0 + \frac{2}{100} \e^2 r_0,
 \label{bounds.tilde.u.tilde.v} 
\end{align}
where $c_0 := \sqrt{2} \tilde c_1$, 
$r_0 := \sqrt{2} \tilde r_1$, 
and $\tilde c_1, \tilde r_1$ are defined 
in Lemma \ref{lemma:chaos.u.without.square}.
The inequalities \eqref{bounds.tilde.u.tilde.v} hold 
for the indices $j$ described in Lemma \ref{lemma:chaos.u}, 
that is, for $j = 0, \ldots, N$, where $N$ satisfies \eqref{cond.N}.
\end{lemma}

\begin{proof} 
By Lemma \ref{lemma:isometry}, one has 
$\mN(t) = \sqrt{2} \| f(t) \|_1$.
Hence \eqref{bounds.tilde.u.tilde.v} follows directly from 
\eqref{bounds.f}.
\end{proof}


\begin{proof}[Proof of Theorem \ref{thm:short}]
All the previous smallness conditions on $\s$ are satisfied 
for $\s \leq \s_4 := \min \{ \s_0, \s_1, \s_2, \s_3\}$, 
where $\s_0 = \s_0(a_0)$ is given by Proposition \ref{prop:chaos.xi.eta}, 
$\s_1$ by Lemma \ref{lemma:maybe.not.useful}, 
$\s_2$ by Lemma \ref{lemma:Sn.eps}, 
$\s_3$ by Lemma \ref{lemma:T.polar}. 
Note that $\s_0, \s_1, \s_2, \s_3$ are all universal constants, 
and therefore $\s_4$ is universal too. 
All the previous smallness conditions on $\e$ are satisfied 
for $0 < \e \leq \e_4$, where $\e_4$ is defined in Lemma \ref{lemma:bound.f.g}.
Also note that $\e_4$ depends only on $m,p$. 

Let $m=2$, and let $p$ be the minimum integer such that $p > m = 2$ 
and $\s = m/p = 2/p \leq \s_*$. 
Since $\s_*$ is a universal constant, the integers $m,p$ are universal constants too. 
Then all the constants 
depending only on $m,p$ now become universal constants. 
In particular, $M_0(\s)$ given by Proposition 
\ref{prop:chaos.xi.eta} is now a universal constant. 

By Lemma \ref{lemma:bound.phys.var}, 
renaming $s_j, \bar s_j$ the times $t_j, \bar t_j$ in \eqref{def:tj},
renaming $\tau$ the period $T = T_a$, at $a = a_0$, in Proposition \ref{prop:chaos.xi.eta},
renaming $t_j, \bar t_j$ the times $t_j^*, \bar t_j^*$ in \eqref{def.rescaled.times}, 
renaming $I_j, E_j$ the intervals $I_j^*, E_j^*$ in \eqref{def.rescaled.times}, 
renaming $\e$ the product $\e c_0$ where $c_0$ is defined in Lemma \ref{lemma:bound.phys.var},
renaming $r_0$ the constant $r_0/(2 c_0^2)$ where $c_0, r_0$ are defined in Lemma \ref{lemma:bound.phys.var},
renaming $b$ the ratio $b/c_0^3$ where $c_0$ is defined in Lemma \ref{lemma:bound.phys.var}
and $b$ in \eqref{def.s.1234.a1.a2.b},
and also renaming $u$ the solution $\tilde u$ in Lemma \ref{lemma:bound.phys.var},
the proof of Theorem \ref{thm:short} is complete. 
\end{proof}

\begin{footnotesize}

\end{footnotesize}

\bigskip

\begin{flushright}
\begin{small}

\textbf{Pietro Baldi}

Dipartimento di Matematica e Applicazioni ``R. Caccioppoli'', Universit\`a di Napoli Federico II

Via Cintia, Monte S. Angelo, 
80126 Napoli, Italy

\texttt{pietro.baldi@unina.it}

\bigskip

\textbf{Filippo Giuliani}

Dipartimento di Matematica, Politecnico di Milano

Piazza Leonardo da Vinci, 32 - Campus Bonardi
20133 Milano, Italy 

\texttt{filippo.giuliani@polimi.it}

\bigskip

\textbf{Marcel Guardia}

Departament de Matem\`atiques i Inform\`atica, Universitat de Barcelona, 

Gran Via, 585, 08007 Barcelona, Spain

Centre de Recerca Matem\`atica

Edifici C, Campus Bellaterra, 08193 Bellaterra, Spain

\texttt{guardia@ub.edu}

\bigskip

\textbf{Emanuele Haus}

Dipartimento di Matematica e Fisica, Universit\`a Roma Tre

Largo San Leonardo Murialdo 1, 
00146 Roma, Italy

\texttt{ehaus@mat.uniroma3.it}

\end{small}
\end{flushright}

\end{document}